\documentclass[12pt]{article}

\usepackage{amsmath}
\usepackage{amssymb}
\usepackage{amsthm}
\usepackage{bbm,mathtools}
\usepackage{enumerate}
\usepackage[dvipsnames]{xcolor}

\newtheorem{thm}{Theorem}[section]
\newtheorem{dfn}[thm]{Definition}
\newtheorem{lem}[thm]{Lemma}
\newtheorem{rem}[thm]{Remark}
\newtheorem{prop}[thm]{Proposition}
\newtheorem{asm}[thm]{Assumption}

\newtheorem{exm}[thm]{Example}
\newtheorem{cor}[thm]{Corollary}

\numberwithin{equation}{section}
\def\reff#1{{\rm(\ref{#1})}}

\def\be{\begin{equation}}
\def\ee{\end{equation}}

\def\bea{\begin{eqnarray}}
\def\eea{\end{eqnarray}}

\def\bea*{\begin{eqnarray*}}
\def\eea*{\end{eqnarray*}}

\def\cB{{\mathcal B}}
\def\cC{{\mathcal C}}
\def\cD{{\mathcal D}}

\def\cF{{\mathcal F}}
\def\cG{{\mathcal G}}
\def\cH{{\mathcal H}}
\def\cI{{\mathcal I}}
\def\cK{{\mathcal K}}
\def\cL{{\mathcal L}}
\def\cM{{\mathcal M}}

\def\cP{{\mathcal P}}
\def\R{{\mathcal R}}

\def\cU{{\mathcal U}}

\def\cQ{{\mathcal Q}}
\def\cT{{\mathcal T}}

\def\cX{{\mathcal X}}

\def\E{\mathbb{E}}
\def\F{\mathbb{F}}

\def\N{\mathbb{N}}

\def\Q{\mathbb{Q}}

\def\R{\mathbb{R}}

\newcommand{\one}{\mathbbm 1}

\def\ve{\varepsilon}
\def\cad{c\`adl\`ag\ }

\def\reff#1{{\rm(\ref{#1})}}
\def\Om{\Omega}
\def\om{\omega}
\def\omh{\widehat{\omega}}
\def\omhn{\widehat{\omega}_n}
\def\hcI{\widehat{\cI}}

\newtheorem{theorem}{Theorem}[section]
\newtheorem{lemma}[theorem]{Lemma}

\title{Martingale optimal transport duality\footnote{Research 
was partly supported by the Swiss National Foundation Grant SNF 200020-172815.}}

\author{Patrick Cheridito\footnote{Department of Mathematics, ETH Zurich, Switzerland}
\and Matti Kiiski\footnotemark[2]
\and David J. Pr{\"o}mel\footnote{Mathematical Institute, University of Oxford, UK}
\and H. Mete Soner\footnote{Department of Operations Research and Financial Engineering, Princeton University, USA}}

\date{}

\begin{document}

\maketitle

\begin{abstract}

We obtain a dual representation of the Kantorovich functional defined for functions on the Skorokhod space using quotient sets. Our representation takes the form of a Choquet capacity generated by martingale measures satisfying additional constraints to ensure compatibility with the quotient sets. These sets contain stochastic integrals defined pathwise and two such definitions starting with simple integrands are given. Another important ingredient of our analysis is a regularized version of Jakubowski's S-topology on the Skorokhod space.
 
\end{abstract}


% Classification
% --------------
% 60B05 - Probability theory on algebraic and topological structures (Probability measures on topological spaces)
% 60G44 - Stochastic processes (Martingales with continuous parameter)
% 91G20 - Mathematical finance (Derivative securities)
% 91B24 - Mathematical economics (Price theory and market structure)

\section{Introduction}

Kantorovich duality 
\cite{kantorovich1,kantorovich2} is an important
tool in the classical theory of optimal
transport \cite{ambrosio-gigli,BK,villani}.
Abstractly it provides a dual
representation for  a 
convex, lower semicontinuous functional
$\Phi$ defined on a {\em{locally convex
Riesz space}} $\cX$, i.e., a 
 locally convex lattice-ordered topological vector space.
In the Kantorovich setting, $\cX$ is a set of real-valued functions defined
on a topological space $\Om$. Typical examples are the set of all bounded continuous
functions $\cC_b(\Om)$ or the set of all bounded Borel measurable functions
$\cB_b(\Om)$ with the supremum norm. 

For a {\em{quotient set}} given by a convex cone $\cI$, we consider the extended real-valued functional given by
\begin{equation*}
\Phi(\xi;\cI) := \inf \left\{ c \in \R : c+ \ell \ge \xi \mbox{ for some } \ell \in \cI \right\},
\quad \xi \in \cX.
\end{equation*} 
There are several immediate properties of
$\Phi$. For instance, it follows directly from the definition that $\Phi$ is monotone and convex.  Also,
it is clear that for any constant~$c$, one has
$\Phi(c;\cI)\le c$ and 
$\Phi(\lambda \xi;\cI)= \lambda \Phi(\xi;\cI)$
for every $\lambda \ge 0$. If additionally, one can establish that
$\Phi$ is lower semicontinuous and proper (i.e.,
not identically equal to infinity and never equal to minus infinity),
then one may apply the Fenchel-Moreau theorem \cite[Theorem~2.4.14]{zalinescu} to obtain the  
representation
\begin{equation} \label{rep}
  \Phi(\xi;\cI)= \sigma_{\partial \Phi}(\xi)
  := \sup_{\varphi \in \partial \Phi}\ \varphi(\xi), \quad \xi \in \cX,
\end{equation}
where the set of sub-gradients  $\partial \Phi$ is 
the convex subset of  the topological dual $\cX^*$ of $\cX$ given by
$\partial \Phi = \partial \Phi(0;\cI):= \left\{ \varphi \in \cX^* :
\varphi(\xi) \le \Phi(\xi;\cI) \mbox{ for all } \xi \in \cX  \right\}$.

This formulation is similar to the one given in \cite{BK}.
In addition to many other applications, it  provides a natural framework 
for risk management \cite{RR,risk}. Recently, it has also been used to reduce 
model dependency in pricing problems \cite{BHLP,paris}. In these applications, $\Phi$ 
is the {\em{super-replication functional}} and $\cI$ the {\em{hedging set}}. The main 
goal of this paper is to establish the dual representation \eqref{rep} in the case where 
$\Omega$ is a suitable subset of the Skorokhod space 
taking also the trajectory of transportation into account.

In classical optimal transport, one has $\Om= \R^d \times \R^d$
and the quotient set $\cI_{ot}$ is defined through
two given probability measures $\mu,\nu$ on $\R^d$
by
$$
\cI_{ot}:=  \left\{ f \oplus h : f, h \in \cC_b(\R^d) \mbox{ and } \mu(f)=\nu(h)=0
\right\},
$$
where $\mu(f)=\int f\, d\mu$, $\nu(h)=\int h\, d\nu$  and
$(f\oplus h) (x,y):= f(x)+ h(y)$.
Let $\Phi_{ot}$ be the corresponding convex functional
on $\cX= \cC_b(\Om)$ with the supremum norm.  
Then, it is immediate that $\Phi_{ot}$ is proper
and Lipschitz continuous. Moreover,
$$
\partial \Phi_{ot} = \left\{ \varphi \in \cC_b(\Om)^* : \varphi \ge 0 \mbox{ and }
\varphi(f\oplus h)= \mu(f)+\nu(h) \mbox{ for all } f, h \in \cC_b(\R^d)
\right\}.
$$
Hence, any  $\varphi \in \partial \Phi_{ot}$ is non-negative and has marginals
$\mu$ and $\nu$.  It follows that $\varphi$ is tight and therefore a Radon 
probability measure on $\Omega$.

Alternatively,
one could deduce the countable additivity
of the dual elements by
using the $\beta_0$-topology on $\cC_b(\Om)$
recalled in Appendix~\ref{sec:beta0} below.
If the topology on $\Om$ is completely regular Hausdorff (T$_{3\frac12}$), then
the topological dual of $\cC_b(\Om)$ with the $\beta_0$-topology
is  equal to the set of all signed Radon measures of finite total variation
and the tightness argument is not needed.  On the other hand, one then has to prove 
the continuity of $\Phi$ with respect to this topology.  We use
this observation in our study, which considers the problem on a more complex 
topological space~$\Om$.

Kellerer \cite{kellerer} used Choquet’s capacibility theorem \cite{choquet} to show that the optimal transport 
duality also holds for measurable (and Suslin) functions if measurable functions are used in the 
definition of the quotient set. Similarly, for martingale or constrained optimal transport, one needs to 
enlarge the set ${\cal I}$ to achieve duality for more general functions with the same set of sub-gradients
\cite{BNT,ekren-soner}. Alternatively, one could fix the quotient set $\cI$
and obtain duality by extending the set of sub-gradients
as it is done in \cite{ekren-soner}. We do not pursue this approach here.

In this paper, we study general martingale optimal transport
on a subset $\Omega$ of the Skorokhod space $\cD([0,T];\R^d_+)$ of all $\R^d_+$-valued
\cad functions, i.e., functions $\om:[0,T] \mapsto \R_+^d$ that are
continuous from the right and have finite left limits. 
We assume that $\Om$ is a closed subset of $\cD([0,T];\R^d_+)$ with respect to Jakubowski's 
$S$-topology \cite{jakubowski, jakubowski4} and endow it with a regularized version of $S$.
Our main goal is to prove duality with the same Choquet capacity
defined by countably additive (martingale) measures,
for different choices of $\cX$ by appropriately extending the quotient set.  

Martingale optimal transport was first introduced
in a discrete time model in \cite{BHLP}
and in continuous time in \cite{paris}.
Since then it has been investigated 
extensively. The initial duality results
\cite{yan-mete,mete3,yan-mete2,obloj}
are proved
by real-analytic techniques and
only for uniformly continuous functions.
Alternatively,
\cite{BCK,BKN,david2,CKT2,CKT} use functional analytical tools.  In particular,
\cite{CKT2} provides
a general representation result. 
\cite{CKT} proves duality in discrete time and \cite{BCK,BKN,david2} for a $\sigma$-compact set
$\Omega$. Our approach is similar to that of \cite{david2}  but
without the assumption of $\sigma$-compactness.
Instead, we use the $S$-topology introduced by Jakubowski \cite{jakubowski}
which provides an efficient
characterization
of compact sets via up-crossings. This characterization
allows us in Theorem~\ref{t.local} to construct an increasing sequence of
compact sets $K_n$ such that $\Phi(\one_{\Omega \setminus K_n})$
decreases to zero.  This  localization result is central to our approach. 
In a similar context, Jakubowski's $S$-topology was first used
in \cite{GTT3,GTT2,GTT1} to prove several important properties of martingale optimal transport.
Their set-up is related to \cite{paris} and differs from ours.

In martingale optimal transport, the quotient set $\cI$
contains the ``stochastic integrals''.  Since
there is no a priori given probabilistic structure,
the definition of the
 integral must be pathwise and is 
a delicate aspect of the problem.
Starting from simple integrands, we first
extend the integrands using
the theory developed by Vovk \cite{vovk,vovk2012},
later by \cite{perkowski} and used
in \cite{david} to prove duality.
This construction provides
duality for upper semicontinuous
functions. 
We then further enlarge the quotient set $\cI$ by taking its Fatou-closure
as defined in Subsection \ref{ss.hs}
and prove the duality for measurable functions
by using Choquet's capacitability theorem
as done earlier in \cite{BCK,BKN,david2,BNT}.
These results are stated in Theorem \ref{t.main}.
Section \ref{sec:counter} provides examples
showing the necessity of enlarging the set of integrands.

There are also deep 
connections between duality and the fundamental 
theorem of asset pricing (FTAP),
which provides equivalent conditions
for the dual set of measures to be non-empty.
In the classical probabilistic setting,
\cite{HP} proves it 
for the Black-Scholes model,
\cite{DMW} for discrete time  and 
\cite{DS,DS1} in full generality.
The robust discrete time model has first been studied in \cite{RFTAP} and
later in \cite{matteo2,matteo,matteo3}. \cite{BBKN,BN} on the other hand
study probabilistic models with none or finitely many static options.
We also obtain a general robust FTAP, Corollary~\ref{c.ftap},
as an immediate consequence of our main duality result Theorem \ref{t.main}.

The paper is organized as follows.  After providing the necessary
structure and definitions in Section \ref{sec:set-up}, we state 
the main result in Section \ref{sec:main}.  Important properties 
of the dual elements are proven in Section \ref{sec:measures},
and several approximation results are derived in Section \ref{sec:approximation}.
Section \ref{sec:Cb} analyses $\Phi$ on $\cC_b(\Om)$.
The proof of the main result, Theorem \ref{t.main}, is given in Section \ref{sec:proofs}.
Several examples are constructed in Section \ref{sec:counter}.
Section \ref{sec:fin} discusses applications to model-free finance.
The topological structures used in the paper and a sufficient condition
for a probability measure to be a martingale measure are given in the Appendix.

\section{Set-up}
\label{sec:set-up}

Let $\Omega$ be a non-empty subset of the Skorokhod space ${\cal D}([0,T]; \mathbb{R}^d_+)$ of 
all c\`adl\`ag functions $\omega \colon [0,T] \to \mathbb{R}^d_+$ that is closed with respect to 
Jakubowski's $S$-topology \cite{jakubowski,jakubowski4}. We denote the relative topology of $S$ on
$\Omega$ again by $S$ and, similarly to \cite{stopo}, endow $\Omega$ with the 
coarsest topology $S^*$ making all $S$-continuous functions $\xi \colon \Omega \to \mathbb{R}$ continuous.
More details on $S$ and $S^*$ are given in Appendix A, where it is shown that $(\Omega, S^*)$ is a 
perfectly normal Hausdorff space ($T_6$), and every Borel probability measure on $(\Omega, S^*)$ 
is automatically a Radon measure. Moreover, we know from \cite{jakubowski, stopo} that for all $s < t$ and 
every $i = 1, \dots, d$, $\int_s^t \omega^i(u)\, du$ is continuous with respect to $S$, and 
\[
\| \omega\|_{\infty} := \sup_{0 \le t \le T} |\omega(t)| 
\]
is $S$-lower semicontinuous, where $|\cdot|$ denotes the Euclidean norm on $\mathbb{R}^d$.

For $t \in [0,T]$, we denote by $X_t(\omega) = \omega(t)$ the coordinate 
map on $\Omega$ and let $\mathbb{F}^X = ({\cal F}_t)_{t \in [0,T]}$ be the natural filtration 
of $X$ given by ${\cal F}^X_t = \sigma(X_s : s \le t)$. By $\mathbb{F} = ({\cal F}_t)$, we denote the
right-continuous filtration given by ${\cal F}_t = {\cal F}^X_{t+} = \bigcap_{s > t} F^X_s$, $t < T$, and
${\cal F}_T = {\cal F}^X_T$. Adapted and predictable processes, as well as stopping times, 
are defined with respect to the filtration $\mathbb{F}$.
In particular, for any open subset $A$ of $\mathbb{R}^d_+$, the hitting time
$\tau_A(\omega) = \inf \{t \ge 0: X_t(\omega) \in A \}$ is a stopping time; see e.g. \cite{dellacheriemeyer,protter} for these facts. 
Moreover, arguments from \cite{jakubowski, stopo} show that ${\cal F}_T = {\cal F}^X_T$ is equal to the collection of all Borel subsets of $(\Omega, S^*)$.

\subsection{Riesz spaces}%\label{ss.riesz}

Let $\cB(\Om)$ be the set of all Borel measurable functions $\xi \colon \Omega \to [-\infty, \infty]$ 
and $\cB_b(\Om)$ the subset of bounded functions in $\cB(\Om)$. For $p \in [1, \infty)$,
we define
$$
\cB_p(\Omega) := \left\{ \xi \in \cB(\Om) : 
\omega \mapsto \xi(\omega)/(1+ \|\om\|_{\infty}^p) \mbox{ is bounded}\right\},
$$
and
$$
\cB_0(\Om):= \left\{ \xi \in \cB_b(\Om) : \mbox{for all } \varepsilon > 0, 
\left\{\omega \in \Omega : |\xi(\omega)| > \varepsilon\right\}
\mbox{ is relatively compact} \right\}.
$$
By $\cU_b(\Om)$ and $\cU_p(\Om)$ we denote the sets of all upper semicontinuous 
functions in $\cB_b(\Om)$ and $\cB_p(\Om)$, respectively. $\cC(\Om)$ is the 
set of all real-valued continuous functions on $\Omega$. $\cC_b(\Om)$ and 
$\cC_p(\Om)$ are defined analogously to $\cU_b(\Om)$ and $\cU_p(\Om)$. 
In addition, we need the set
$$
\cC_{q,p}(\Om):= \left\{ \xi \in \cC(\Om) : \xi^+ \in \cC_q(\Om), \, \xi^-\in \cC_p(\Om)\right\},
$$
where $\xi^+ = \max(\xi,0)$ and $\xi^- = \max(-\xi,0)$.

By ${\cal M}(\Omega)$ we denote the set of all signed Radon measures of bounded total variation on 
$\Omega$ and by $\cP(\Om) \subset \cM(\Om)$ the subset of probability measures.
For $Q \in \cP(\Om)$ and $\xi \in {\cal B}(\Omega)$, we define the expectation $\E_Q[\xi] \in [-\infty,\infty]$ 
by $\E_Q[\xi] := \E_Q[\xi^+]- \E_Q[\xi^-]$ with the convention $\infty-\infty=-\infty$. For $p \ge 1$, 
$\cL^p(\Om,Q)$ is the collection of all functions $\xi \in \cB(\Om)$ satisfying
$\mathbb{E}_Q[|\xi|^p]<\infty$.

The $\beta_0$-topology on $C_b(\Om)$ is generated by the semi-norms 
$\|. \, \eta\|_\infty$, $\eta \in \cB_0^+(\Om)$, where we use the
superscript $^+$ to indicate the subset of non-negative elements.
More details on the $\beta_0$-topology are given in Appendix \ref{sec:beta0}.
Since $(\Om,S^*)$ is a perfectly normal Hausdorff space, it is also completely regular, 
and it follows that the dual of $\cC_b(\Om)$ with the $\beta_0$-topology is $\cM(\Om)$; see e.g.
\cite{jarchow, sentilles}. 

\subsection{The standing assumption}%\label{ss.assumption}

We fix  {\em{universal constants}} $ 1 \le p <q$.  All our definitions and results depend on 
them, but we do not show this dependence in our notation.

\begin{dfn}
{\rm For a convex cone $\cG \subset \cB(\Om)$, we denote by 
$\cQ(\cG)$ the (possibly empty) set of all probability measures 
$Q \in \cP(\Om)$ such that $\E_Q[\gamma]  \le 0$ for all $\gamma \in \cG$
and the canonical map $X$ is an $(\F,Q)$-martingale, i.e., 
for every $t \in [0,T]$, $X_t \in \cL^1(\Om,Q)$
and $\E_Q\left[Y \cdot (X_t-X_T)\right]=0$
for all $\cF_t$-measurable
$Y \in \cB_b(\Om)^d$.}
\end{dfn}

The following assumption is used throughout the paper.
Although all results assume it, we do not always state this assumption 
explicitly.

 \begin{asm}
 \label{asm.1} 
{{ 
$\cG \subset \cC_{q,p}(\Om)$ is a convex cone, and there exist $c_q \in \mathbb{R}_+$ and $\xi_q \in \cG$
such that $\left|X_T\right|^q \le c_q + \xi_q$.
}}
\end{asm}
Then,  for every  $Q \in \cQ(\cG)$, $\E_Q\left[ \left|X_T\right|^q\right]
\le c_q + \E_Q[\xi_q] \le c_q$. We combine this with Doob's martingale 
inequality to conclude that
\be
\label{e.power}
c^*_q:= \sup_{Q \in \cQ(\cG)}
\ \E_Q\left[ X_*^q \right] <\infty, \quad
{\text{where}}\quad
X_*:=\sup_{t \in [0,T]}  |X_t|.
\ee
In particular, $\E_Q\left[Y \cdot (X_t-X_T)\right]=0$ for all $Q \in \cQ(\cG)$, 
every $t \in [0,T]$ and any ${\cal F}_t$-measurable 
$Y \in \cB_{q-1}(\Om)^d$.

If, for a given $\mu \in \cP(\R_+^d)$, $\cG$ contains all functions
$g(\om(T))- \mu(g)$ with $g \in \cC_b(\R_+^d)$,
then any element $Q \in \cQ(\cG)$ has the marginal
$\mu$ at the final time $T$.  Hence, the above construction
includes the classical example of given marginals.
For this example, the celebrated result of Strassen \cite{strassen}
provides necessary and sufficient conditions for $\cQ(\cG)$ to be non-empty;
see also Corollary \ref{c.ftap}, below.

\subsection{Integrals and quotient sets}
\label{ss.hs}
  
A \emph{simple integrand}  $H$ consists of a sequence of pairs $(\tau_n, h_n)_{n\in \N}$
such that
$\tau_0\leq \tau_1\leq \tau_2 \leq \cdots $ are
$\F$-stopping times, and each
$h_n \in \cB_{q-1}(\Om)^d$ is $\cF_{\tau_n}$-measurable.
We assume that for every $\om \in \Om$ there is an index
$n(\om)$ such that
$\tau_{n(\om)}  \ge T$.
The corresponding integral is defined directly as
$$
(H\cdot X)_t(\omega):= 
\sum_{n=0}^\infty h_n(\omega)\cdot
(X_{\tau_{n+1}\wedge t}(\omega)-X_{\tau_n\wedge t}(\omega)),
\quad (t,\om) \in [0,T] \times \Om.
$$
A simple integrand $H$ is  called 
\emph{admissible} if for some $\lambda \in \cB^+_q(\Om)$
$$
(H\cdot X)_{\tau_m \wedge t} (\omega)\geq -\lambda(\omega) \quad \mbox{for all }
(t,\omega,m)\in[0,T]\times\Omega\times\mathbb{N}.
$$
$\cH_s$ denotes the set of all admissible 
simple integrands.
An {\em{admissible integrand}} is a 
collection of simple integrands  $H:=(H^k)_{k \in \N}\subset \cH_s$
satisfying
$(H^k \cdot X)_t \ge - \Lambda$,
for every  $t \in [0,T]$, $k\in \N$,
for some $\Lambda \in \cB_q^+(\Om)$.  $\cH$ denotes
the set of all admissible integrands.  The corresponding  
integral is defined pathwise by,
$$
(H \cdot X)_t(\om):= \liminf_{k \to \infty} (H^k \cdot X)_t(\om) \quad 
\mbox{for all } (t,\om) \in [0,T]\times\Om.
$$
We use the following quotient sets:
$$
\cI_s(\cG):= \{ \gamma + (H \cdot X)_T :
 \gamma \in \cG,\  H \in \cH_s\},
 $$
 $$
\cI(0):= \{ (H \cdot X)_T:  H \in \cH\},
\quad
\cI(\cG):= \{\gamma + (H \cdot X)_T :
 \gamma \in \cG,\  H \in \cH\}.
$$
Moreover, let $\hcI(\cG)\subset \cB(\Om)$ be the Fatou-closure of $\cI(\cG)$, i.e., the smallest set 
of extended real-valued Borel measurable functions containing $\cI(\cG)$ with the property that
for every sequence $\{\ell_n\}_{n \in \N} \subset \hcI(\cG)$ satisfying a 
uniform lower bound $\ell _n \ge -\lambda$ for some $\lambda \in \cB_q^+$,
$\liminf_n \ell_n \in \hcI(\cG)$. In the context of financial applications,
similar integrals were first constructed in \cite{vovk} and later used in 
\cite{david,perkowski,vovk2012}. Their properties have recently been studied in \cite{fatou}. 

It is clear that $\cI_s(\cG) \subset
\cI(\cG) \subset \hcI(\cG)$
and $\cI(0)$  are all convex cones.
 
\section{Main result}
\label{sec:main}

\begin{theorem}
\label{t.main}
Under Assumption \ref{asm.1},
\begin{align}
\label{e.dual1}
\Phi(\xi;\cI(\cG))& = \sigma_{\cQ(\cG)}(\xi) \quad \mbox{for all } \xi \in \cU_p(\Om) \mbox{ and}\\
\label{e.dual2}
\Phi(\xi;\hcI(\cG)) &= \sigma_{\cQ(\cG)}(\xi) \;  \mbox{ for all } \xi \in \cB_p(\Om),
\end{align}
where $\sigma_{\cQ(\cG)}(\cdot)
:=\sup_{Q \in \cQ(\cG)}\E_Q[\cdot]$
is the support functional of $\cQ(\cG)$.
\end{theorem}

The proof is given in Section \ref{sec:proofs}. If $\cQ(\cG)$ is empty,
by convention $\sigma_{\cQ(\cG)}$ is identically equal to minus infinity, in which case
both sides of the above equalities are 
equal to minus infinity; see Corollary \ref{c.alt}.
Counter-examples
of Section \ref{sec:counter}
show that  in general  $\cI_s(\cG)$ could be smaller
than $\cI(\cG)$ and \reff{e.dual2} does not hold in general with $\cI(\cG)$.
 
\section{Properties of $\cQ(\cG)$}
\label{sec:measures}
Recall $X_*$ in \reff{e.power}.  If  $\cQ(\cG)$ is empty, all results of this section  hold trivially.
For $\xi \in \cB(\Om)$, we define for every constant $c \ge 0$, 
\be
\label{e.xic}
\xi^c(\om):= (c \wedge \xi(\om))  \vee (-c), \quad  \om \in \Om.
\ee

\begin{lem}
\label{l.est} 
$\lim_{c \to \infty} \sigma_{\cQ(\cG)}(\xi^c)
=\sigma_{\cQ(\cG)}(\xi)$ for all $\xi \in \cB_p(\Om)$.
\end{lem}

\begin{proof}
Fix $Q \in \cQ(\cG)$ and $\xi \in \cB_p(\Om)$.
There exists
a constant $c_0>0$ so that
$|\xi(\om)| \le c_0 X_*^p(\om)$ whenever $|\xi(\om)| \ge c_0$.
Using \reff{e.power},
we  estimate that for $c\ge c_0$,
\begin{align*}
 \E_Q[\left|\xi -\xi^c\right|]
& \le  \E_Q[|\xi| \one_{\{|\xi| \ge c\}}]
\le c_0 \E_Q[X^p_* \one_{\{X^p_* \ge c/c_0\}}]\\
&\le c_0 \frac{\E_Q[X^q_* \one_{\{X^p_* \ge c/c_0\}}]}{(c/c_0)^{q/p - 1}}
 \le \frac{c_0^{q/p} c^*_q}{c^{q/p -1}}.
\end{align*}
Hence, by sub-additivity,
$$
\left| \sigma_{\cQ(\cG)}(\xi) - \sigma_{\cQ(\cG)}(\xi^c) \right| 
\le 
\sigma_{\cQ(\cG)}\left(\left|\xi -\xi^c\right|\right)
\le \sup_{Q\in \cQ(\cG)}\ \E_Q[\left|\xi -\xi^c\right|]
\le \frac{c_0^{q/p} c^*_q}{c^{q/p-1}}.
$$
\end{proof}

\begin{lem}
\label{l.mart}  
For every $H \in \cH$, $t \in [0,T]$, and $Q \in \cQ(\cG)$,
$\E_Q\left[ (H \cdot X)_t \right] \le 0$.
Consequently, $\E_Q[\ell] \le 0$
for every  $\ell \in \cI(\cG)$ and $Q \in \cQ(\cG)$.
\end{lem}

\begin{proof}  Fix $Q \in \cQ(\cG)$ and
$H=(\tau_n,h_n)_{n \in \N}\in \cH_s$.  
For $m \ge 1$, set
$$
\ell^{m}_t:=\left(H\cdot X\right)_{\tau_m \wedge t}
=\sum_{n=0}^{m-1} h_n \cdot
(X_{\tau_{n+1}\wedge t}-X_{\tau_n\wedge t}),
\quad t \in [0,T].
$$
Since by definition each $h_n \in \cB_{q-1}(\Om)^d$ 
and $X$ is an $(\F,Q)$-martingale, we have
$E_Q[\ell^{m}_t] =0$.  By the admissibility of $H$,
there exists $\lambda \in \cB_q^+(\Om)$ such that
$\ell^m_t \ge -\lambda$ for each $m$ and $ t \in [0,T]$.
Therefore, by Fatou's Lemma and \reff{e.power},
$$
\E_Q \left[ (H \cdot X)_t \right] \le 
\liminf_{m \to \infty} \E_Q\left[ \ell^{m}_t \right] =0
\quad \mbox{for all } t \in [0,T].
$$
Let $H=(H^k)_{k \in \N}\in \cH$.  Then, by definition each $H^k \in \cH_s$ and
by the above result $\E_Q \left[ (H^k \cdot X)_t \right] \le 0$.
Again by admissibility, there exists $\Lambda \in \cB_q^+(\Om)$ so that
$(H^k \cdot X)_t \ge -\Lambda$ for each $k \ge 1, t \in [0,T]$. 
By Fatou's Lemma,
for  $t \in [0,T]$,
$$
\E_Q \left[ (H \cdot X)_t \right] 
=  \E_Q \left[ \liminf_{k \to \infty}(H^k \cdot X)_t \right] \le
\liminf_{k \to \infty} \E_Q\left[ (H^k \cdot X)_t \right]   \le 0.
$$
The final
statement follows directly from the definitions.
\end{proof}

\begin{lem}%\label{l.fatou-mart} 
For every $\ell \in \hcI(\cG)$ and $Q \in \cQ(\cG)$, 
$\E_Q[\ell] \le 0$. Therefore,
\be
\label{e.lower}
\sigma_{\cQ(\cG)}(\xi)
\le \Phi(\xi;\hcI(\cG)) \le \Phi(\xi;\cI(\cG)) \quad \mbox{for all } \xi \in \cB(\Om).
\ee
\end{lem}
\begin{proof}
Set 
$\cK(\cG) := \left\{ \xi \in \cB(\Om) :
\sigma_{\cQ(\cG)}(\xi) \le 0 \right\}$.
By Lemma \ref{l.mart}, $\cI(\cG) \subset \cK(\cG)$.
Consider a sequence $\{\xi_n\}_{n\in \N} \subset \cK(\cG)$
satisfying a uniform lower bound $\xi_n \ge -\lambda$
for some $\lambda \in \cB_q^+(\Om)$.
Then, by Fatou's Lemma and the uniform bound
\reff{e.power}, 
$$
\E_Q\left[ \liminf_{n \to \infty} \xi_n\right]
\le \liminf_{n \to \infty}  \E_Q[\xi_n] \le0 \quad \mbox{for all } Q \in \cQ(\cG).
$$
Hence, $\liminf_n \xi_n \in \cK(\cG)$.
Since $\hcI(\cG)$, by its definition,
is the smallest set of measurable functions
with this property
containing $\cI(\cG)$,
we conclude that $\hcI(\cG) \subset \cK(\cG)$.

Fix $\xi \in \cB(\Om)$.
Suppose that $\xi \le c +\ell$ for some $c \in \R$ and
$\ell \in \hcI(\cG)$.  Since $\hcI(\cG) \subset \cK(\cG)$,
$\E_Q[\xi] \le  \E_Q[c + \ell ] \le c$ for every
$Q \in \cQ(\cG)$.
Hence, $\sigma_{\cQ(\cG)}(\xi) \le c$.
Since $\Phi(\xi;\hcI(\cG))$
is the infimum of all such constants,
$\sigma_{\cQ(\cG)}(\xi)\le \Phi(\xi; \hcI(\cG))$.
The fact $\cI(\cG) \subset \hcI(\cG)$
implies that
$ \Phi(\xi; \hcI(\cG)) \le  \Phi(\xi; \cI(\cG))$.
\end{proof}

\section{Approximation results}
\label{sec:approximation}

\begin{lem}
\label{dslemma} 
$\hat{c}^*_q:=\Phi(X_*^q;\cI(\cG))<\infty$.
\end{lem}

\begin{proof} For $N \in \N$,
$\{y_k\}_{k =0}^N\subset \R_+$ and $n \le N$, let
$y^*_n:=   \max_{0\leq k\leq n}y_k$.

{\em{Step 1.}}  
It is shown in \cite[Proposition~2.1]{schachermayer} that
$$
%\label{e.doob}
(y_N^*)^q + d_qy_0^q
\leq \sum_{n=0}^{N-1} h (y^*_n)(y_{n+1}-y_n)
+\left(d_q y_N\right)^q,
$$
where $d_q:= q/(q-1)$ and $h(y):=-qd_qy^{q-1}$
for $y \in \R_+$.

{\em{Step 2.}}
Set $\tau_0:=0$ and for each $\om \in \Om$ and $n\in \N$
define recursively
$$
\tau_n(\omega):=\inf\left\{t>\tau_{n-1}(\omega) : |X_t(\om)|^q > |X_{\tau_{n-1}}(\om)|^q + 1 \right\} \wedge T.
$$
Then, the
$\tau_n$'s are stopping times.
For $\om \in \Om$, $i = 1,\ldots,d$, $n=0,1,2,\dots$, set
$$
h^{*,i}_n(\om):=
h\left(\max_{0\leq k\leq n}\ X^i_{\tau_k}(\om) \right),
\quad 
h^*_n(\om):=
\left(h^{*,1}_n(\om),\dots,
h^{*,d}_n(\om)\right). 
$$
It is clear that $h^*_n \in \cB_{q-1}(\Om)^d$
and therefore  $H^*:=(\tau_n,h^*_n)_{n \in \N}$ is a simple integrand.

{\em{Step 3.}} We claim that $H^*$ is
admissible.  Indeed, fix $t \in [0,T]$,  $\om \in \Om$, $i =1, \dots, d$, and set
$y_n:= X^i_{\tau_n\wedge t}(\omega)$.  For $k \in \N$, set
$$
\tilde{n}=\tilde{n}(\om,t,k):= 
\sup\left\{ m : \tau_m(\om) \leq  t\right\}
\wedge (k-1).
$$ 
Then, for 
$n \le  \tilde{n}$, $y_n=X_{\tau_n}^i$ and therefore,
$h_n^{*,i}(\om) = h(y^*_n)$. For $\tilde{n} < n <k$,
$X^i_{\tau_{n+1}\wedge t} 
=X^i_{\tau_{n}\wedge t} = X^i_t$ and $y_{\tilde{n}+1}
=X^i_{\tau_k\wedge t}$.
By Step 1,
\begin{align*}
\sum_{n=0}^{k-1} h^{*,i}_n
(X^i_{\tau_{n+1}\wedge t}-X^i_{\tau_n\wedge t})
&= \sum_{n=0}^{\tilde{n}} h(y_n^*) (y_{n+1} -y_n)
\ge (y^*_{\tilde{n}+1})^q -(d_q y_{\tilde{n}+1})^q
\\
& = \sup_{ n \le k} \left( X^i_{\tau_n\wedge t}\right)^q
 -(d_q X^i_{\tau_k\wedge t})^q.
\end{align*}
Hence, for every $ t \in [0,T]$
and integer $k$,
\begin{align}
\nonumber
(H^*\cdot X)_{\tau_k\wedge t}  &
\ge 
\sum_{i \le d} \sup_{ n \le k} \left( X^i_{\tau_n\wedge t}\right)^q
- \sum_{i \le d}  (d_qX^i_{\tau_k \wedge t})^q\\
\label{e.15}
&\ge 
\sum_{i \le d} \sup_{ n \le k} \left( X^i_{\tau_n\wedge t}\right)^q
-  c^*\ \left|X_{\tau_k \wedge t}\right|^q
\ge -  c^* X_*^q,
\end{align}
for some constant $c^*$ depending only on $d$ and $q$.
Hence, $H^*$  is admissible.

{\em{Step 4.}} We let $t=T$ in \reff{e.15}
and send  $k$ to infinity to obtain
$$
\sum_{i \le d} \sup_{ n} \left( X^i_{\tau_n\wedge T}\right)^q
\le
(H^*\cdot X)_{T} + c^*\ \left|X_T\right|^q.
$$
Choose a constant $\hat{c}^*$ so that
for  all $y=(y^1,\ldots,y^d) \in \R_+^d$,
$ |y|^q \le \hat{c}^*  \sum_{i \le d} |y^i|^q$.
Let $c_q, \xi_q$ be as in  Assumption \ref{asm.1}.
Then, 
\begin{align*}
0\le \sup_{ n} \left| X_{\tau_n \wedge T}\right|^q
& \le  \hat{c}^*  
\sum_{i \le d} \sup_{ n } \left( X^i_{\tau_n\wedge T}\right)^q
\le \hat{c}^* \left[ (H^*\cdot X)_T + c^*\left|X_T\right|^q\right]\\
& \le  
(\hat{c}^* H^*\cdot X)_T+ \hat{c}^* c^* (c_q + \xi_q)  =: \ell^* +\hat{c}^* c^*c_q.
\end{align*}
Since $H^* \in \cH_s$, $\xi_q \in \cG$ and $\cG$ is a cone, 
$\ell^*:=  (\hat{c}^* H^*\cdot X)_T +\hat{c}^*  c^* \xi_q \in \cI(\cG)$.

{\em{Step 5.}} By definition of the $\tau_n$'s, 
$$
X_*^q \le \sup_n\left|X_{\tau_n\wedge T}\right|^q + 1 \le \ell^* + \hat{c}^* c^* c_q+ 1,
$$
from which one obtains $\Phi(X_*^q;\cI(\cG)) \le \hat{c}^* c^* c_q + 1 < \infty$.
\end{proof}

\begin{cor}
\label{c.approximate}
For any convex cone $\cI \supset \cI(\cG)$ and $\xi \in \cB_p(\Om)$, one has
$\lim_{c \to \infty} 
\Phi(\xi^c; \cI)=  \Phi(\xi;\cI)$.
\end{cor}
\begin{proof}   
Fix  $\xi \in \cB_p(\Om)$.  There exists $c_0>0$ so that 
$|\xi| \le c_0 X_*^p$ whenever $|\xi| \ge c_0$. 

{\em{Step 1.}}
For $c \ge c_0$, 
\begin{align*}
(|\xi| -c) \one_{\{|\xi| \ge c\}} & \le  c_0 X_*^p \ \one_{\{ X_*^p  \ge c/c_0\}}
\le  \frac{c_0 X_*^q}{(c /c_0)^{q/p-1}} \one_{\{ X_*^p  \ge c/c_0\}} \le \frac{c_0^{q/p}}{c^{q/p-1}} X_*^q.
\end{align*}
Since $\cI$ includes $\cI(\cG)$, $\Phi((|\xi| -c) \one_{\{|\xi| \ge c\}};\cI) \le 
c_0^{q/p} c^{1 - q/p}\Phi(X_*^q;\cI(\cG))$, which in view of Lemma \ref{dslemma}, gives
$\limsup_{c \to \infty}\ \Phi((|\xi| -c) \one_{\{|\xi| \ge c\}};\cI) \le 0$.

{\em{Step 2.}}
Since $|\xi - \xi^c| \le 
(|\xi| -c) \one_{\{|\xi|  \ge c\}}$, one obtains from sub-additivity,
$$
\Phi(\xi^c;\cI) \le
 \Phi( \xi^c-\xi;\cI)+\Phi(\xi;\cI)  \le 
\Phi((|\xi| -c) \one_{\{|\xi| \ge c\}};\cI) +\Phi(\xi;\cI),
$$
which by the previous step, gives
$\limsup_{c \to \infty}\
\Phi( \xi^c ;\cI)  \le 
\Phi(\xi;\cI)$.

{\em{Step 3.}}
Similarly,
$$
\Phi(\xi;\cI) \le \Phi( \xi-\xi^c;\cI)+\Phi(\xi^c;\cI)  \le 
\Phi((|\xi| -c) \one_{\{|\xi| \ge c\}};\cI) + \Phi(\xi^c;\cI),
$$
and therefore, $\Phi(\xi;\cI) \le \liminf_{c \to \infty} \Phi( \xi^c ;\cI)$.
\end{proof}

It is a direct consequence of the definition that $\Phi(\xi; \cI(\cG)) \le \|\xi\|_\infty$ for any $\xi \in \cB_b(\Om)$. 
In particular, $\Phi(0; \cI(\cG)) \le 0$.

\begin{cor} \label{c.alt}
We have the following alternatives:
 \begin{itemize}
 \item[{\rm (i)}]
If $\Phi(0; \cI(\cG))=0$, then
$\left|\Phi(\xi; \cI(\cG))\right| \le \|\xi\|_\infty$ for all $\xi \in \cB_b(\Om)$.
\item[{\rm (ii)}]
If $\Phi(0; \cI(\cG))<0$, then $\cQ(\cG)$ is empty,
and $\Phi(\cdot;\cI(\cG)) \equiv\Phi(\cdot;\hcI(\cG))\equiv -\infty$
on $\cB_p(\Om)$.
In particular,
\reff{e.dual1}  and \reff{e.dual2} hold trivially.
\end{itemize}
\end{cor}
\begin{proof}
First, suppose that $\Phi(0; \cI(\cG))=0$ and let $\xi \in \cB_b(\Om)$. Since
$\xi + \|\xi\|_\infty \ge 0$, one has
$$
\Phi(\xi; \cI(\cG)) = - \|\xi\|_\infty+ \Phi(\xi+ \|\xi\|_\infty; \cI(\cG))
 \ge 
- \|\xi\|_\infty+ \Phi(0; \cI(\cG)) = - \|\xi\|_\infty.
$$

Now assume that $\Phi(0; \cI(\cG))<0$.  Then, there exist $c<0$, $\ell  \in \cI(\cG)$ such that
$c+\ell \ge 0$.  Also, for any constant
$\lambda >0$, $\lambda (c+ \ell) \ge 0$. Since $\cI(\cG)$ is a 
cone, $  \lambda \ell \in \cI(\cG)$ and consequently,
 $\Phi(0; \cI(\cG)) \le c \lambda$.  As $\lambda>0$ above was arbitrary,
we have $\Phi(0; \cI(\cG))=-\infty$ and 
$$
\Phi(\xi; \cI(\cG)) \le \|\xi\|_\infty+ \Phi(\xi- \|\xi\|_\infty; \cI(\cG)) \le
\|\xi\|_\infty+ \Phi(0 ; \cI(\cG)) =- \infty.
$$
This shows that $-\infty\leq \sigma_{\cQ(\cG)}(\cdot) \le \Phi(\cdot;\hcI(\cG)) \le \Phi(\cdot; \cI(\cG)) 
\equiv -\infty$ on $\cB_b(\Om)$, and by Corollary~\ref{c.approximate}, also on 
$\cB_p(\Om)$.

Moreover, \reff{e.lower} implies that if $\cQ(\cG)$ is non-empty,
$\Phi(0; \hcI(\cG))=0$. Hence if $\Phi(0; \hcI(\cG))<0$, $\cQ(\cG)$ must be empty.
\end{proof}

For an $\R^d$-valued c\`adl\`ag process $Y$, set
\begin{equation*}%\label{e.int}
\ell_Y(\om):= \int_0^T \ Y_u(\om) \cdot (X_u(\om)-X_T(\om))\ du.
\end{equation*}

\begin{lem}
\label{c.integral}
Let $Y$ be  an $\R^d$-valued, adapted, c\`adl\`ag process.
Suppose that there exists $\lambda \in \cB_{q-1}(\Om)$
satisfying $|Y_u| \le {\lambda}$ for every
$u \in [0,T]$.  
Then, $\ell_Y \in \cI(0)$ and
for any quotient set $\cI$ containing
$\cI(0)$, $\Phi(\ell_Y; {\cI}) \le 0$
\end{lem}
\begin{proof} 
For $k\in \N$ and $n=0,\ldots,k$ set
$\tau^k_n:= n T/k$, $Y^k_n:= Y_{\tau^k_n}$,
$X^k_n:= X_{\tau^k_n}$, 
 $h^k_0:=-(T/k)Y_0$ and $h_n^k:= h^k_{n-1}-(T/k)Y^k_n $ for $n \ge 1$.
Since $\lambda \in  \cB_{q-1}(\Om)$, 
the simple integrand $H^k:=(\tau^k_n,h^k_n)_{n=0}^{k}$
is admissible.
Moreover,
$$
(H^k\cdot X)_T= \sum_{n=0}^{k-1} h^k_n \cdot (X_{n+1}^k -X_{n}^k)
= \frac{T}{k}\, \sum_{n=0}^{k-1} Y^k_n \cdot (X_{n}^k-X_T).
$$
Let $H:=(H^k)_{k\in \N}$.  Since both $Y$ and $X$ are c\`adl\`ag,
$$
(H\cdot X)_T = \lim_{k \to \infty} (H^k\cdot X)_T =\ell_Y.
$$
One can directly verify that $H \in \cH$.  Hence, $\ell_Y \in \cI(0)$.
\end{proof}

\section{Continuity on $\cC_b(\Om)$}
\label{sec:Cb}

We use the compact notation
$\Phi(\cdot)= \Phi(\cdot;\cI(\cG))$.

\begin{lem}\label{l.bound} 
$\limsup_{c \to \infty} \ \Phi(\one_{\{ X_*  >   c\}})
\le 0$.
\end{lem}

\begin{proof}
Fix $c>0$, $i\in\{1,\dots,d\}$ and
set $X^i_*:= \sup_{t \in [0,T]} X^i_t$.
Since $X^i_T \ge 0$,
$$
c \one_{\{ X^i_*  >   c\}}(\om)
\leq X^i_T(\om) 
+(c-X^i_T(\om)) \one_{\{ X^i_* >  c\}}(\om) \quad \mbox{for all } \om \in \Om.
$$
Set $h_1^i =-1$, $h^j_1=0$ for $j\neq i$,
$\tau_0(\omega)
:=\inf\{t \ge 0 : X^i_t\in (c,\infty)\}
\wedge T$, $\tau_1:=T$ and let $H$ be the corresponding
integrand.  Then, $(H\cdot X)_T=  (X^i_{\tau_1}-X^i_T)$.
By right-continuity, we have $X^i_{\tau_1} \ge c$ on the set $\{X^i_* >  c\}$ 
and $\tau_1=T$ on its complement. Consequently,
$(H\cdot X)_T \ge  (c-X^i_T) \one_{ \{ X^i_* > c\}}$
and
$\Phi\left((c-X^i_T)\ \one_{\{ X^i_*>  c\}}\right) \le 0$.
Therefore,
$\Phi(\one_{\{  X^i_* >   c\}})  \le   \ \Phi\left(X^i_T\right)/c$, which, by Lemma \ref{dslemma}, converges to
zero as $c$ tends to infinity.  
Since $\{X_*>\sqrt{d} c\} \subset \cup_i \{  X^i_* >   c\}$,
the claim of the lemma follows from the sub-additivity of $\Phi$.
\end{proof}

\begin{dfn}
\label{d.up}
{\rm{For $\om \in \cD([0,T];\R_+)$, $t \in [0,T]$, and
$a<b$, the number of}} up-crossings up to $t$,
$U_t^{a,b}(\om),$ {\rm{is the largest 
integer
$n$ for which one can find 
$0\leq t_1<\cdots<t_{2n}\leq t$ 
such that $\om(t_{2k-1}) <  a$ and $\om(t_{2k}) > b$ 
for $k=1,\dots,n$. }}

\end{dfn}

For $\omega\in\cD([0,T];\mathbb{R}_+^d)$, 
we set $U_t^{a,b,i}(\om):=U_t^{a,b}(\om^i)$.

\begin{lem}
\label{l.vovk}
For $0< a<b$ and $i=1,\ldots,d$, 
there exists  $H^{a,b,i} \in \cH_s$ 
such that
$$
(H^{a,b,i}\cdot X)_t(\om) \geq -a + (b-a) 
U^{a,b,i}_t(\om) \quad \mbox{for all }  (t,\om) \in [0,T]\times \Om.
$$
\end{lem}

\begin{proof}
For $k \ge 1$, set $I_k:=[0,a)$ if $k$ is an odd integer  
and $I_k:=(b,\infty)$ if $k$ is even, and $\tau_0:= 0$.
Recursively define a sequence of random times  by
$$
\tau_k(\omega):=\inf\left\{t\geq \tau_{k-1}(\om)\, :\, X^i_t(\om) \in I_k
\right\} \wedge T,
$$
where the infimum over an empty set is infinity.
Since  $X$ is c\`adl\`ag and $I_k$ is open,  
$\tau_k$'s are $\F$-stopping times. 
Define $h_k=(h^1_k,\ldots,h^d_k)$ as follows:
$h^i_k:=1$ when $k$ is odd, $h^i_k:=0$ for $k$ even and $h^j_k=0$
for $j\neq i$.
Let  $H^{a,b,i}$ be the corresponding simple integrand.
It is clear that for every $t \in [0,T]$, $\om \in \Om$,
$(H^{a,b,i}\cdot X)_t(\om) \ge - a
+(b-a)U_t^{a,b,i}(\om)$. Hence, $H^{a,b,i} \in \cH_s$. 
\end{proof}

\subsection{Localization}%\label{ss.lsc} 

\begin{thm}
\label{t.local} There exists an increasing sequence of compact subsets 
$\{K_n\}_{n \in \N}$ of $\Om$  satisfying,
$$
\lim_{n \to \infty} \Phi(\one_{\Omega \setminus K_n}) \le 0.
$$
\end{thm}

\begin{proof}  We complete the proof in several steps.

{\em{Step 1}}.
Let $D$ be a countable dense subset of $(0,\infty)$ and $\{(a_j,b_j) : j \in \N\}$ an enumeration 
of the countable set $\{(a,b) \in D \times D : 0<a<b\}$. For all $n \in \mathbb{N}$, define
$$
K^{i,j}_n:= \left\{\om \in \Om : U^{a_j,b_j,i}_T(\omega)\leq c^j_n 
\right\}, \ \ \widehat{K}_n:= \bigcap_{i=1}^d\bigcap_{j\in \N}\ K^{i,j}_n, \ \
K_n:= B_n \cap \widehat{K}_n,
$$
where $c^n_j:=2^{j+n} (a_j \vee 1)/(b_j-a_j)$ and $B_n:=\{\omega \in \Om : X_*(\om) \leq n\}$.
Since $\Omega$ is $S$-closed, one obtains from \cite[Corollary 2.10]{jakubowski} that
$K^{ij}_n$ and $B_n$ are S-closed subsets of  $\cD([0,T];\R^d_+)$. Hence, all $K_n$ are $S$-compact 
and therefore also $S^*$-compact subsets of $\Omega$; see Appendix \ref{s.jakubowski} or \cite[Corollary~5.11]{stopo}. 
Moreover,
$$
\left(\Om \setminus K_n \right)\subset
O_n  \ \cup \ (\Om \setminus B_n), \quad
{\text{where}}\quad
O_n:= \bigcup_{i,j}\ (\Om \setminus K^{i,j}_n).
$$

{\em{Step 2}}.  
Let $H^{a,b,i}$ be as in  Lemma \ref{l.vovk}
and set $ H^{i,j}_{n}:= (c^n_j(b_j-a_j))^{-1} H^{a_j,b_j,i}$.
Then, for every $t \in [0,T]$,
$$
(H^{i,j}_n \cdot X)_t \ge
-\frac{a^j}{c^n_j(b_j-a_j)}+ \frac{U^{a_j,b_j,i}_t}{c^n_j}
\ge -2^{-(j+n)}+ \frac{U^{a_j,b_j,i}_t}{c^n_j}.
$$
Hence, $H^{i,j}_n \in \cH_s$ and also
$(H^{i,j}_n \cdot X)_T \ge
- 2^{-(j+n)}+ \one_{\Om \setminus K_n^{i,j}}$.

For $k\ge 1$, set
$H^k_n :=  \sum_{i=1}^d  \sum_{j=1}^k\ H^{i,j}_n$.
Then, for every $k \ge 1$ and $t \in [0,T]$,
$(H^k_n \cdot X)_t\ge - d \ 2^{-n}$.  
Hence, for each $n$, $H_n:=(H^k_n)_{k\in \N} \in \cH$ and 
$$
(H_n \cdot X)_T=\liminf_{k \to \infty} (H^k_n \cdot X)_T\ge  
\sum_{i=1}^d \sum_{j=1}^\infty
\left(\one_{\Omega \setminus K^{i,j}_n}-2^{-(j+n)}\right) \ge \one_{O_n} - d \ 2^{-n}.
$$
Therefore,  $\Phi(\one_{O_n}) \le d\ 2^{-n}$.
\vspace{4pt}

{\em{Step 3.}} By the previous steps and Lemma \ref{l.bound},
$$
\limsup_{n \to \infty} \Phi(\one_{\Omega \setminus K_n}) \le 
\limsup_{n \to \infty} \left(\Phi(\one_{\Omega \setminus B_n}) + \Phi(\one_{O_n}) 
\right)\le 0.
$$

Finally, since for each pair $(i,j)$, the sets $K^{i,j}_n$
are increasing in $n$, we conclude that $K_n$ is also increasing in $n$.
\end{proof}

\subsection{$\beta_0$-continuity}%\label{ss.dualcb}

\begin{prop}%\label{p.lsc} 
Suppose that
$\Phi(0) =0$. Then $\Phi$ is real-valued  and $\beta_0$-continuous on $\cC_b(\Om)$.
\end{prop}
\begin{proof} 
By Corollary \ref{c.alt}, $\Phi$ is real-valued and the compact sets constructed in
Theorem~\ref{t.local} 
satisfy $\Phi(\one_{\Omega \setminus K_n}) \downarrow 0$
as $n$ tends to infinity.  
Let $K_0$ be the empty set and by re-labelling, 
we may assume that
$\Phi(\one_{\Omega \setminus K_k}) \le 2^{-2k}$,
for every $k\ge 0$.
Define
$$
\eta^*:=  \sum_{k=1}^\infty \ 2^{-k} \one_{K_k\setminus K_{k-1}}.
$$
Since  on the complement of $K_{k-1}$,
$|\eta^*| \le 2^{-k}$, $\eta^* \in \cB_0(\Om)$.  
Fix an integer $n$ and
$\xi \in \cC_b(\Om)$.  Since
on $K_k\setminus K_{k-1}$, $\eta^* =2^{-k}$, on $K_k\setminus K_{k-1}$, $(\eta^*)^{-1} =2^{k}$, so, on $K_n= \cup_{k=1}^{n}(K_k\setminus K_{k-1})$,
$$
|\xi| \one_{K_n}= |\xi| \eta^*\ (\eta^*)^{-1}\one_{K_n}
 \le \|\xi\eta^*\|_\infty\ (\eta^*)^{-1}\one_{K_n}=  
\|\xi\eta^*\|_\infty\
\sum_{k=1}^{n} 2^k \ \one_{K_k\setminus K_{k-1}}.
$$
In view of the hypothesis $\Phi(0)=0$, 
$|\Phi(\xi\one_{K_n})| \le \Phi(|\xi|\one_{K_n})$ and
consequently,
\begin{align*}
\left|\Phi\left(\xi \one_{K_n}\right)\right|
&\le \|\xi\eta^*\|_\infty 
\sum_{k=1}^{n} 2^k\Phi\left(\one_{K_k\setminus K_{k-1}}\right)
\le \|\xi\eta^*\|_\infty 
\sum_{k=1}^{n} 2^k\Phi\left(\one_{\Om\setminus K_{k-1}}\right)\\
&\le \|\xi\eta^*\|_\infty  \sum_{k=1}^{n} 2^k 2^{-2(k-1)} \le 4  \|\xi\eta^*\|_\infty
= 4 \|\xi\|_{\eta^*}.
\end{align*}
Therefore, by Theorem~\ref{t.local},
$$
|\Phi(\xi)| \leq\limsup_{n\rightarrow\infty}
\left(|\Phi(\xi\one_{K_n})|
+\|\xi\|_\infty\Phi(\one_{\Omega\setminus K_n})\right)
\leq 4 \|\xi\eta^*\|_\infty.
$$

For $\xi, \zeta \in \cC_b(\Om)$, 
by sub-additivity,
$\Phi(\xi)= \Phi((\xi-\zeta)+\zeta) \le \Phi(\xi-\zeta) + \Phi(\zeta)$.
Hence, $\Phi(\xi) - \Phi(\zeta)
\le \Phi(\xi- \zeta) \le 4 \|(\xi - \zeta) \ \eta^*\|_\infty$.  
Switching the roles of $\xi$ and $\zeta$, we conclude that
$\left| \Phi(\xi) - \Phi(\zeta)\right|
\le 4 \|(\xi - \zeta) \ \eta^*\|_\infty$.
Since the $\beta_0$-topology is generated by the semi-norms 
$\| \,.\,  \eta\|_\infty$ for arbitrary $\eta \in \cB^+_0(\Om)$ and 
$\eta^* \in \cB^+_0(\Om)$, the above inequality yields that $\Phi$ is $\beta_0$-continuous on 
$\cC_b(\Om)$ (see Appendix \ref{sec:beta0} below).

\end{proof}

\subsection{Sub-differential}%\label{ss.sd}

\begin{prop}
\label{p.subdifferential}
$\cQ(\cG)=\partial\Phi:=\left\{\varphi \in \cM(\Om):
\varphi(\xi) \leq \Phi(\xi;\cI(\cG)),\  \xi \in \cC_b(\Om)\right\}$.
\end{prop}

\begin{proof}
The lower bound \reff{e.lower} implies that 
$\cQ(\cG) \subset \partial\Phi$.  
To prove the opposite inclusion, fix
$Q \in \partial \Phi \subset \cM(\Om)$.
The monotonicity of $\Phi$ implies that
$Q\ge 0$.  Since $\Phi(c) \le c$ for every constant $c$,
we conclude that
 $Q \in \cP(\Om)$.
 
 {\em{Step 1.}} 
Let $\xi \in \cC^+(\Om)$, and define $\xi^c$ for $c \ge 0$ as in \reff{e.xic}.
Then, $\xi^c \le \xi$ and  by the defining property of $Q$, 
$\E_Q[\xi^c] \le\Phi(\xi^c) \le \Phi(\xi)$.
So, by monotone convergence,
$\E_Q[\xi] = \lim_{c \to \infty} \E_Q[\xi^c]
\le \Phi(\xi)$.

 {\em{Step 2.}} 
For $\ve>0$, set 
$$
X^\ve_t(\om):= \frac{1}{\ve} \int_{t}^{t+\ve} 
X_{u\wedge T} \ du, \ t \in [0,T].
$$
Since the map $X_T$ and time integrals are
$S$-continuous \cite{jakubowski},
$X^\ve$ is $S^*$-continuous.
Hence, for every $t \in [0,T]$,  $|X_t^\ve| \in \cC_1(\Om)$ and
$\left|X^\ve_t\right| \le X_*$, where $X_*$ is 
as in \reff{e.power}.
Also, $\lim_{\ve \to 0} X^\ve_{t}(\om)=X_t(\om)$
for every $\om \in \Om$. 
Fix $t \in [0,T]$ and choose $\xi =|X_t^\ve|^q$ 
in Step 1
 to obtain,
$\E_Q[|X_t^\ve|^q] \le \Phi(|X_t^\ve|^q) \le
\Phi(X_*^q)$.
By Fatou's Lemma,
\begin{equation*}
%\label{e.qmoment}
\E_Q[|X_t|^q] \le \liminf_{\ve \to 0} \E_Q[|X_{t}^\ve|^q]
\le \Phi(X^q_*) = \hat c^ *_q<\infty,
\end{equation*}
where $\hat c^ *_q$ is as in Lemma \ref{dslemma}.
Hence, $X_t \in \cL^q(\Om,Q)$ for every $t \in [0,T]$.

{\em{Step 3.}} 
Fix $t \in [0,T)$ and an $\cF_t$-measurable $Y \in \cC_b(\Om)^d$.
For $\ve \in (0,T-t]$, set
$$
 \ell_{Y,\ve}:=
\frac{1}{\ve} \int_t^{t+\ve} 
\frac{Y}{|X_u|+1} \cdot (X_u-X_T)\, du,\quad
{\text{and}}\quad
\ell_Y:=\frac{Y}{|X_t|+1} \cdot (X_t-X_T).
$$
Observe that $\ell_{Y,\ve} \in \cC_1(\Om)$,
$\lim_{\ve \to 0}
\ \ell_{Y,\ve}(\om)=
\ell_Y(\om)$, for all $\om \in \Om$  and
in view of 
Corollary \ref{c.integral}, $\Phi(\ell_{Y,\ve}) \le0$.
Moreover,
$$
\ell_{Y,\ve} \ge - \|Y\|_\infty \ [ 1+ |X_T|]
\quad
\Rightarrow
\quad
\ell_{Y,\ve}^c \ge - \|Y\|_\infty \ [ 1+ |X_T|] \in \cL^q(\Om,Q).
$$
Then, by Fatou's Lemma, 
$\E_Q[\ell_Y ] \le \liminf_{\ve \to 0} \E_Q[ \ell_{Y,\ve}]$.
We now use again   Fatou's Lemma,
the sub-differential inequality,
and Corollary \ref{c.approximate} to obtain
$\E_Q[\ell_{Y,\ve}] \le \liminf_{c \to \infty} \E_Q[ \ell_{Y,\ve}^c]
\le  \lim_{c \to \infty} \Phi( \ell_{Y,\ve}^c) = \Phi(\ell_{Y,\ve})\le 0$.
Since this  argument also holds for $-Y$, we conclude that
$\E_Q [\ell_Y]= 0$.

{\em{Step 4.}}
 Let $Y$ be as in the previous step.
For $c>0$, set $Y_c:= Y [(|X_t|+1)\wedge c]$.  
Since $X_t, X_T \in \cL^q(\Om,Q)$, by dominated convergence,
$$
\E_Q[Y \cdot(X_t-X_T)] 
= \lim_{c\to \infty}
\E_Q\left[\frac{Y_c}{|X_t|+1} \cdot (X_t-X_T)\right]  =0.
$$
The above equality, the integrability proved 
in Step 2 and Lemma \ref{l.martingale}
imply that $X$ is an $(\F,Q)$-martingale.
As in \reff{e.power}, this also implies that
$\E_Q \left[X_*^q \right] <\infty$.
  
{\em{Step 5.}}  Let $\xi \in \cC_p(\Om)$. Then, $|\xi|\le c_\xi (1+X_*^q)$
for some constant $c_\xi$ and for every $c>0$,
 $|\xi^c|\le c_\xi X_*^q$.
Since $X_* \in \cL^q(\Om,Q)$,
dominated convergence yields that 
$\E_Q[\xi] = \lim_{c \to \infty} \E_Q[\xi^c]$.
Also, by Corollary \ref{c.approximate},
$\lim_{c \to \infty} \Phi(\xi^c)= \Phi(\xi)$
and the sub-differential inequality at $\xi^c \in \cC_b(\Om)$
imply that $\E_Q[\xi^c] \le \Phi(\xi^c)$.
Hence,
$\E_Q[\xi] = \lim_{c \to \infty} \E_Q[\xi^c]
\le \lim_{c \to \infty} \Phi(\xi^c)= \Phi(\xi)$
for every $\xi \in \cC_p(\Om)$.

{\em{Step 6.}} Fix $ \gamma \in \cG$.
Then, by Assumption \ref{asm.1},
$\gamma \in \cC_{q,p}(\Om)$ and hence,
$\gamma^-\in \cC_p(\Om)$.
For $a>0$, set $\gamma_a:= 
\gamma \wedge a$.
Since $\gamma_a \le \gamma$, 
$\Phi(\gamma_a) \le \Phi(\gamma) \le 0$.
Also, $\gamma_a \in \cC_p(\Om)$
and by the previous step,
$\E_Q[\gamma_a] \le \Phi(\gamma_a) \le 0$.
Moreover, $\left|\gamma_a\right| \le \left|\gamma\right|
\le c_\gamma (1+X_*^q)$ for some $c_\gamma>0$.
Since  $X_* \in \cL^q(\Om,Q)$,
dominated convergence yields
$\E_Q[\gamma]=\lim_{a \to \infty} \E_Q[\gamma_a] \le 0$.
Hence, $\E_Q[\gamma] \le 0$ for every $\gamma \in \cG$.  This 
and Step 4 imply that
$\partial \Phi \subset \cQ(\cG)$.

\end{proof}

The above results also prove
the compactness of the set
$\cQ(\cG)$.

\begin{cor}
\label{c.compact}
$\cQ(\cG)$ is convex and both compact as well as sequentially compact with respect to the topology induced by the pairing 
$\langle\cC_b(\Om),\cM(\Om)\rangle$.
\end{cor}

\begin{proof}
It is clear that $\cQ(\cG)$ 
is convex.
Let $K_n$ be 
as in the 
proof of Theorem~\ref{t.local}.
Then, by \eqref{e.lower},
$\sigma_{\cQ(\cG)}(\one_{\Om \setminus K_n}) 
\le \Phi(\one_{\Om \setminus K_n})=: \alpha_n$.
By Theorem \ref{t.local}, $\alpha_n$ tends to zero. Hence, 
$Q(K_n)\ge 1-\alpha_n$ uniformly over $Q \in \cQ(\cG)$.
Since $\alpha_n$ converges to zero, 
$\cQ(\cG)$ is uniformly tight.

By Proposition \ref{p.subdifferential},
$\cQ(\cG)=\bigcap_{\xi \in \cC_b(\Om)} \left\{Q \in \cP(\Om) \ : \ \E_Q[\xi] \le \Phi(\xi)\right\}$.
Hence, $\cQ(\cG)$  is weak${}^*$ closed.
Then, by Prokhorov's theorem for completely regular Hausdorff spaces 
\cite[Theorem~8.6.7]{bogachev}, $\cQ(\cG)$ is weak${}^*$ compact. By the second assertion 
of \cite[Theorem~8.6.7]{bogachev}, since the compact sets $K_n$ above are 
metrizable \cite[Proposition~5.7]{stopo}, $\cQ(\cG)$ is also sequentially weak${}^*$ compact.
\end{proof}

\section{Proof of Theorem \ref{t.main}}
\label{sec:proofs}

\subsection{Duality on $\cC_b(\Om)$}%\label{sec:d}

\begin{prop}%\label{p.cdual}
$\Phi(\xi;\cI(\cG))=\sigma_{\cQ(\cG)}(\xi)$ for all $\xi \in \cC_b(\Om)$.
\end{prop}
\begin{proof}
In view of Corollary \ref{c.alt},
we may assume that $\Phi(0;\cI(\cG))=0$.
Then, by the results of
Section \ref{sec:Cb}, $\Phi$ is convex, finite-valued and $\beta_0$-continuous.
Hence, the hypotheses of the Fenchel-Moreau
theorem on the topological space $\cC_b(\Om)$ with 
the locally convex $\beta_0$-topology are
satisfied
\cite[Theorem~2.3.3]{zalinescu}.  Since $\Phi$
is positively homogenous, 
$\Phi(\xi;\cI(\cG))= \sigma_{\partial \Phi}(\xi)$
for every $\xi \in \cC_b(\Om)$. 
We then
complete the proof of
duality on $\cC_b(\Om)$
by Proposition \ref{p.subdifferential}.
\end{proof}

\subsection{Duality on $\cU_p(\Om)$}%\label{sec:d1}

We first extend the duality from $\cC_b(\Om)$ to $\cU_{b}(\Om)$ 
by a minimax argument.  

\begin{lem}
\label{l.extension1}
The duality $\Phi(\xi;\cI(\cG))=\sigma_{\cQ(\cG)}(\xi)$
holds for all  $\xi \in \cC_b(\Om)$ if and only if it holds  for all  $\xi \in\cU_{b}(\Om)$.
\end{lem}

\begin{proof}
Assume that the duality holds on $\cC_b(\Om)$ and 
let $\eta\in\cU_{b}(\Om)$.  
In view of  \reff{e.lower}, we need to show that
$\sigma_{\cQ(\cG)}(\eta) \geq \Phi(\eta;\cI(\cG))$.
Since $S^*$ is perfectly normal by Lemma \ref{l.matti} below,
for every   $Q \in \cP(\Om)$,
$\E_Q[\eta]=\inf_{\eta \leq \xi \in \cC_b(\Om)}\ \E_Q[\xi]$.
Clearly, $\{\xi \in \cC_b(\Om) : \eta \leq \xi\}$ is a convex subset 
of $\cC_b(\Om)$ and the mapping 
that takes $(\xi,Q)$
to $\E_Q[\xi]$ is continuous and bilinear on 
$\cC_b(\Om) \times \cQ(\cG)$. 
Moreover, by Corollary~\ref{c.compact}, $\cQ(\cG)$ is a 
convex, weak$^*$ compact subset of $\cP(\Om)$. 
Hence, the assumptions of a standard minimax argument are satisfied,
see e.g. \cite[Theorem~2.10.2]{zalinescu}. 
Since $\Phi$ is monotone,
\begin{align*}
\sigma_{\cQ(\cG)}(\eta) 
& = \sup_{Q\in\cQ(\cG)}\ 
\inf_{\eta \leq \xi \in \cC_b(\Om)}\ \E_Q[\xi]
= \inf_{\eta \leq \xi \in \cC_b(\Om)}\ 
\sup_{Q\in\cQ(\cG)}\ \E_Q[\xi]\\
&=  \inf_{\eta \leq \xi \in \cC_b(\Om)}\ 
\sigma_{\partial \Phi}(\xi)
=  \inf_{\eta \leq \xi \in \cC_b(\Om)}\ 
\Phi(\xi;\cI(\cG))\geq\Phi(\eta;\cI(\cG)).
\end{align*}
Therefore, the duality holds on $\cU_{b}(\Om)$.
\end{proof}

\begin{prop}
\label{p.udual}
$\Phi(\xi;\cI(\cG))=\sigma_{\cQ(\cG)}(\xi)$ for all $\xi \in \cU_p(\Om)$.
\end{prop}

\begin{proof}
Fix $\xi \in \cU_p(\Om)$ and $\xi^c$ be as in \reff{e.xic}.  Then, 
$\xi^c \in \cU_b(\Om)$ and duality holds at $\xi^c$.  We now combine this
with Lemma~\ref{l.est} and Corollary \ref{c.approximate} to arrive at
$$
\sigma_{\cQ(\cG)}(\xi)=
\lim_{c \to \infty} \sigma_{\cQ(\cG)}(\xi^c)
=\lim_{c \to \infty} \Phi(\xi^c;\cI(\cG))
=\Phi(\xi;\cI(\cG)).
$$
\end{proof}

\subsection{Duality on $\cB_p(\Om)$}%\label{sec:d2}

In this section, we follow the approach of \cite{kellerer}
and extend the duality
to measurable functions
by the Choquet capacitability theorem
\cite{choquet}.

\begin{prop}
$\Phi(\xi;\hcI(\cG))= \sigma_{\cQ(\cG)}(\xi)$ for all $\xi \in \cB_p(\Om)$.
\end{prop}
\begin{proof}
We write $\widehat{\Phi}(\cdot)$
instead of $\Phi(\cdot;\hcI(\cG))$
and $\Phi(\cdot)$
for $\Phi(\cdot;\cI(\cG))$ as before.

{\em{Step 1.}}
Since $\hcI(\cG) \supset \cI(\cG)$,
$\widehat{\Phi} \le \Phi$.  By Proposition \ref{p.udual}
and  \reff{e.lower},
for every $\eta \in \cU_b(\Om)$,
$\sigma_{\cQ(\cG)}( \eta)
\le \widehat{\Phi}(\eta) \le \Phi(\eta)
=\sigma_{\cQ(\cG)}( \eta)$.  Hence,
$\Phi=\widehat{\Phi}$ on $\cU_b(\Om)$.

{\em{Step 2.}}  
Consider a sequence $\{Q_n\}_{n \in \N}$ in $\cM(\Om)$
converging  to $Q^*$
in the weak$^*$ topology.
Then,
$\E_{Q_n}[\xi]$ converges to $\E_{Q^*}[\xi]$
for every $\xi \in \cC_b(\Om)$.
Fix $\eta \in \cU_b(\Om)$.  
Since $S^*$ is perfectly normal  by Lemma \ref{l.matti} below,
there is a decreasing sequence
$\{\xi_k\}_{k \in \N} \subset \cC_b(\Om)$ converging to $\eta$ and 
$\E_{Q^*}[\xi_k]$ converges  to $\E_{Q^*}[\eta]$.
We use this and the weak$^*$ convergence of 
$Q_n$ to arrive at
$$
\limsup_{n \to \infty} \E_{Q_n}[\eta] \le 
\inf_{k}\ 
\lim_{n \to \infty} \E_{Q_n}[\xi_k] = 
\inf_{k}\ 
\E_{Q^*}[\xi_k] = \E_{Q^*}[\eta].
$$
Since by  Corollary \ref{c.compact},
$\cQ(\cG)$ is weak$^*$ compact,
the above property
implies that for every $\eta \in \cU_b(\Om)$
there is $Q_\eta \in \cQ(\cG)$ satisfying,
$\E_{Q_\eta}[\eta]=\sigma_{\cQ(\cG)}(\eta)$.

{\em{Step 3.}}  
Suppose that a sequence $\{\eta_n\}_{n \in \N} \subset \cU_b(\Om)$
decreases monotonically to a function $\eta^* \in \cU_b(\Om)$. 
Then, $Q_n:= Q_{\eta_n}$ satisfies 
$\E_{Q_n}[\eta_n]=\sigma_{\cQ(\cG)}(\eta_n)$.
Since $\cQ(\cG)$ is sequentially compact with respect to $\sigma(\cM,\cC_b)$,
there is a subsequence
(without loss of generality,
again denoted by $Q_n$)
and $Q^* \in \cQ(\cG)$
such that $Q_{n}$ converges
to $Q^*$ in the weak$^*$ topology.  
Then, by the previous step,
$$
\limsup_{n \to \infty} \E_{Q_n}[\eta_n]  \le
\inf_k\ \limsup_{n \to \infty} \E_{Q_n}[\eta_k]
 \le \inf_k\ \E_{Q^*}[\eta_k]
 = \E_{Q^*}[\eta^*],
$$
where we used monotone convergence in the final equality. 

By the first step, $\Phi=\widehat{\Phi}$ on $\cU_b(\Om)$.
Then, by Proposition \ref{p.udual} and \reff{e.lower},
\begin{align*}
\limsup_{n \to \infty} \widehat{\Phi}(\eta_n)  
& = \limsup_{n \to \infty}\sigma_{\cQ(\cG)}(\eta_n)   
= \limsup_{n \to \infty} \E_{Q_n}[\eta_n]  
\\&
\le  \E_{Q^*}[\eta^*]
\le\sigma_{\cQ(\cG)}(\eta^*) 
\le \widehat{\Phi}(\eta^*).
\end{align*}
Since $\eta_n$'s are decreasing to $\eta^*$,
the opposite inequality is immediate.
Hence,
\be
\label{e.down}
\lim_{n \to \infty} \widehat{\Phi}(\eta_n) = \widehat{\Phi}(\eta^*) 
\quad {\text{whenever}}\quad
 \cU_b(\Om) \ni \eta_n \downarrow \eta^* 
 \in \cU_b(\Om) \quad
{\text{as}}\ n \to \infty.
\ee

{\em{Step 4.}}  
Consider $\{\zeta_n\}_{n \in \N} \subset \cB_b(\Om)$
increasing monotonically to $\zeta^*\in \cB_b(\Om)$.  
Choose $\{\ell_n\}_{n \in \N} \subset \hcI(\cG)$ so that
$\widehat{\Phi}(\zeta_n) + \frac1n+ \ell_n(\om)
\ge \zeta_n(\om)$, 
for every $ \om \in \Om$.
It is clear that $\widehat{\Phi}(\zeta_1) \le \widehat{\Phi}(\zeta_n) \le \widehat{\Phi}(\zeta^*)$.
Since $\zeta_n \ge \zeta_1$,
$\ell_n \ge (\zeta_1 - \widehat{\Phi}(\zeta^*) - 1) \wedge 0  =: -\lambda$.
Then, by the definition of $\hcI(\cG)$,
 $\ell^*:= \liminf_{n} \ell_n \in \hcI(\cG)$.  Therefore,
 $$
\zeta^*(\om) = \lim_{n \to \infty} \zeta_n(\om)
\le  \liminf_{n \to \infty} \left[\widehat{\Phi}(\zeta_n) + \frac1n+ \ell_n(\om)\right]
= \lim_{n \to \infty}\widehat{\Phi}(\zeta_n) +\ell^*(\om),
$$
for every $\om \in \Om$.
Hence, $ \lim_{n \to \infty} \widehat{\Phi}(\zeta_n) \ge \widehat{\Phi}(\zeta^*)$.  Again
the opposite inequality is immediate.  So we have shown that
\be
\label{e.up}
\lim_{n \to \infty} \widehat{\Phi}(\zeta_n) =\widehat{\Phi}(\zeta^*) 
\quad {\text{whenever}}\quad
\cB_b(\Om) \ni \zeta_n  \uparrow \zeta^* \in \cB_b(\Om)\quad
{\text{as}}\ n \to \infty.
\ee

{\em{Step 5.}}  \reff{e.down} and \reff{e.up} imply
that we can apply  the Choquet capacitability theorem 
(see \cite[Proposition~2.11]{kellerer} or \cite[Proposition~2.1]{BCK})
to the functional $\widehat{\Phi}$. Let $\mathcal{S}(\Omega)$ denote the family of all Suslin functions 
generated by $\cU_b(\Omega)$ i.e. functions of the form
$\sup_{\phi\in\mathbb{N}^\mathbb{N}}\inf_{k\geq 1} \xi_{\phi\mid k}$, where 
$\phi| k$ denotes the restriction of $\phi \in \mathbb{N}^\mathbb{N}$ to 
$\{1, \dots, k\}$ and each $\xi_{\phi\mid k}$ is an element of $ \cU_b(\Omega)$; we refer to 
\cite[Section~42]{fremlin} for details. Since the $S^*$-topology on $\Om$ is
perfectly normal, 
by Lemma \ref{l.matti} below,
the family $\mathcal{S}(\Om)$ 
contains $\cB_b(\Om)$.
Moreover, $\widehat{\Phi}=\Phi$ on $\cU_b(\Om)$.
Hence, 
$\widehat{\Phi}(\zeta) = \sup \left\{
\Phi(\eta) : \eta \in \cU_b(\Om),  \eta \le \zeta
\right\}$ for every $ \zeta \in \cB_b(\Om)$.
This approximation
together with the duality proved in Lemma \ref{l.extension1} yield,
\begin{align*}
\widehat{\Phi}(\zeta)&=
\sup_{\{\eta \le \zeta, \
\eta \in \cU_b(\Om)\}} \ \sup_{Q\in \cQ(\cG)} \E_Q[\eta]
=
\sup_{Q\in \cQ(\cG)}\
\sup_{\{\eta \le \zeta, \ 
\eta \in \cU_b(\Om)\}}\  \E_Q[\eta]\\
&= \sup_{Q\in \cQ(\cG)}\ E_Q[\zeta] \quad \mbox{for all } \zeta \in \cB_b(\Om).
\end{align*}
Hence, the duality holds on $\cB_b(\Om)$.

{\em{Step 5.}} We now follow the proof of 
Proposition \ref{p.udual} {\em{mutatis mutandis}}
to extend the result to $\cB_p(\Om)$.
\end{proof}

\section{Counter-examples}
\label{sec:counter}

In  this section, $d=1$, $T=1$, $p=1$ and $q=2$.
For a given   $\mu \in \cP(\R_+)$, set
$$
\cG_\mu:=\left\{ g(X_1(\om)) : g \in \cC_{2,1}(\R_+) \, , \, \mu(g)=0 \right\}.
$$

\begin{exm}%\label{ex.counter2}
{\rm Suppose that $\mu$ is supported in $[1,3]$
and let $\Om =  \cD([0,1];[1,3])$.
Then, there exists a countable set $A \subset \Om$
such that
 $0=\sigma_{\cQ(\cG_\mu)}(\one_A)
= \Phi(\one_A; \cI(\cG_\mu))$
and $\Phi(\one_A; \cI_s(\cG_\mu))=1$. In particular,
$\cI_s(\cG_\mu) \neq \cI(\cG_\mu)$.}
\end{exm}

\begin{proof}
For $\om \in \Om$, let $v_3(\om):=
 \sup_{\pi}\sum_{k=1}^n|\omega(\tau_k)-\omega(\tau_{k-1})|^3$,
where $\pi$ ranges over all finite partitions $0= \tau_0< \tau_1<\dots< \tau_n=1$ of $[0,1]$.
Set $t_0=1$ and for $k\ge 1$,
$t_k:=1/k$, $s_k:=(t_{k+1}+t_k)/2$,
$c_k:= 2f(s_k)/(t_k-t_{k+1})$ with $f(x):=x^{1/3}$ for $x\geq 0$, and
$$  
\omh(t):= 
\sum_{k=1}^\infty\
c_k(t-t_{k+1})\one_{( t_{k+1}, s_k]}(t) +
[f\left(s_k\right)-c_k\left(t-s_k\right)] \one_{(s_k,t_k]}(t),
\quad t \in [0,1].
$$
It is clear that $\omh\in A \subset \Om$
and  $\omh(t_n)=0$, $\omh(s_n)=f(s_n)$.
Set 
$\omhn (t)
:= \omh({t \wedge t_n})$  for $t \in [0,1]$.  
Then, $\omhn(t)=0$ for $t \in [t_n,1]$ and
$$
v_3(\omhn)
\ge \sum_{k=n}^\infty(f(s_k)-f(t_k))^3
=\sum_{k=n}^\infty s_k=\infty.
$$

Let  $\Q$ be the set of 
rational numbers in $[1,2]$.  Set
$A_q:= \left\{ q+ \omhn : n \in \N\right\}$ and
$A:= \cup_{q \in \Q} A_q$.
Then, $A \subset \{\om \in \Om : v_3(\om)=\infty\}$.
Since for any martingale measure  $Q$, $Q(\om \in \Om : v_3(\om)=\infty)=0$,
we conclude that $\sigma_{\cQ(\cG_\mu)}(\one_A)=0$.
Suppose that  for some $c \in \R$,
$\gamma \in \cG$ and $H=(\tau_m,h_m)_{m \in \N} \in \cH_s$,
we have $c+\gamma(\om)+(H \cdot X)_1(\om)\geq \one_A(\om)$
for every $\om \in \Om$. 
Then, $\gamma(\om)= g(X_1(\om))$ with
$\mu(g)=0$ and $\gamma(\om)= g(q)$ for every $\om \in A_q$.  Hence,
$c+g(q)+(H \cdot X)_1(q+\omhn) \geq 1$, 
for every $q \in \Q, n\ge 1$.
By the adaptedness of $H$, 
$(H\cdot X)_1(q+\omhn) = (H\cdot X)_{t_n}(q+\omh)$,
for each $q$.  Therefore,
$$
\lim_{n \to \infty} \ (H\cdot X)_1(q+\omhn) = 
\lim_{n \to \infty} \  (H\cdot X)_{t_n}(q+\omh) =0.
$$
This implies that $c+g(q) \ge 1$.   Moreover, $g$ is continuous
and $\mu(g)=0$. Hence, $c \ge 1$.  
Since  $\Phi(\one_A; \cI_s(\cG_\mu))$ is the smallest
of all such constants, we conclude that $\Phi(\one_A; \cI_s(\cG_\mu)) \ge 1$.
As $\one_A \le 1$, $\Phi(\one_A; \cI_s(\cG_\mu)) = 1$. 

We next proceed
as in Lemma \ref{l.vovk} to show that $\Phi(\one_A; \cI(\cG_\mu)) = 0$.
Indeed, for $k \ge 2$, define $H^k=(\tau^k_m,h^k_m)_{m \in \N} \in \cH_s$
as follows. 
Let $\tau^k_0=0$ and  
for $m \ge 1$, recursively define the stopping times by,
\begin{align*}
\tau^k_{2m-1}(\om)&:= \inf\{ t >\tau^k_{2m-2}(\om) :
\om(t)> \om(0)+ k^{-1/3}/2\}\wedge 1, \\
\tau^k_{2m}(\om)&:= \inf\{ t >\tau^k_{2m-1}(\om) :
\om(t)< \om(0)+ k^{-1/3}/3\}\wedge 1.
\end{align*}
For $m \ge 0$, set $h^k_{2m}= k^{-4/3}$, $h^k_{2m+1}=0$.
Let $U^k_t(\om)$ be the crossings in the time
interval $[0,t]$ between the lower boundary 
$\om(0)+ k^{-1/3}/3$ and the upper boundary $\om(0)+ k^{-1/3}/2$.
Then, as in Lemma~\ref{l.vovk},
$$
(H^k\cdot X)_t(\om) \ge -\frac{\om(0)}{k^{4/3}} + \frac{1}{6 k^{5/3}} U^k_t(\om) \quad
\mbox{for all } t\in [0,1] \mbox{ and } \om \in \Om.
$$
In particular, $H^k \in \cH_s$.  Observe that $U^k_t(q+\omhn) \ge k-n$ for all $n\le k$.
Therefore, for any $q \in [1,2]$,
$$
(H^k\cdot X)_t(q+\omhn) \ge -\frac{2}{k^{4/3}} + \frac{ (k-n)}{6 k^{5/3} } \quad \mbox{for all } 1\le n\le k.
$$
For $\ve>0$, let $H^\ve:=\ve (\widehat{H}^j)_{j \in \N}$, 
where $\widehat{H}^j:= \sum_{k \le j} H^k$.  Then, for each $j\ge 1$,
$$
(H^j\cdot X)_t(\om)=\ve \sum_{1\le k \le j} (H^k\cdot X)_t(\om) \ge 
- \ve \sum_{1 \le k}   \frac{2}{k^{4/3}} =:-\ve C_* \quad \mbox{for all } t\in [0,1].
$$
Hence, $H^\ve$ is admissible.  Also, for  $q \in [1,2]$,
$$
(H^\ve \cdot X)_t(q+\omhn)= \liminf_{j \to \infty} \sum_{1 \le k \le j} \ve (H^k\cdot X)_t(q+\omhn)
 \ge  \sum_{1 \le k } -\frac{2\ve}{k^{4/3}} + \frac{ \ve (k-n)^+}{6 k^{5/3} } =\infty.
$$
Therefore, $\Phi(\one_A;\cI(\cG_\mu)) \le  \ve C_*$ for every $\ve >0$.
\end{proof}

The following example motivates the use of the 
Fatou-closure $\hcI(\cG)$ to establish the duality for measurable functions in 
Theorem~\ref{t.main}. 

\begin{exm}%\label{ex.counter}
{\rm Let $\Om= \cD([0,1];\R_+)$ and consider the
quotient spaces given by $\cG = \{ g(X_1(\om)) : g \in \cC_{2,1}(\R_+),   g(1)=0\}$.
Then, there exists an open set $B \subset \Om$ such that
$0= \sigma_{\cQ(\cG)}(\one_B)= \Phi(\one_B; \hcI(\cG))
<1=\Phi(\one_B; \cI(\cG))$. In particular,
$\cI(\cG) \neq \hcI(\cG)$.}
\end{exm}

\begin{proof}
Consider the  open set
$B:=\{X_T\neq 1\}$, set $\omega^*\equiv 1$
and let $Q^*$ be 
the Dirac measure at $\omega^*$.
Then, $\cQ(\cG)=\{Q^*\}$. Hence, 
$\sigma_{\cQ(\cG)}(\one_B)=
\E_{Q^*}[\one_B] = 0$.

Suppose that $\ell \in \cI(\cG)$ and $c \in \R$ satisfy
$c+\ell \ge \one_B$.  By the definition of $\cI(\cG)$,
there are $H \in \cH$ and $g(X_1(\cdot)) \in \cG$ such that
$\ell(\om)= (H\cdot X)_1(\om)+ g(X_1(\om))$.
Consider a constant path $\om \equiv x $.  Then, 
for this path $(H\cdot X)_T(\omega) =0$ and therefore,
$\one_B(\om)=1 \le c+g(x)$ for every $x \neq 1$.  Since $g(1)=0$
and $g$ is continuous, we conclude that $c \ge 1$.  Hence,
$\Phi(\one_B;\cI(\cG))=1$.
\end{proof}

\section{Financial applications}
\label{sec:fin}

In this section we assume that $X$ models the discounted prices of $d$ assets.
Alternatively, one could also model undiscounted prices and introduce an additional 
process representing a savings account. But this does not change the 
essential mathematical structure; see \cite{CKT}. For related examples and 
discussions of the role of $\Omega$ as a prediction set, we refer to \cite{BKN,david2,obloj}.

The set $\cG$ represents the set of net outcomes of investments in liquid derivative instruments.  
Their initial prices are normalized to zero. Since we do not assume any probabilistic structure, 
this set plays an essential role in determining the pricing functionals.
We give different examples of the set $\cG$. They show that 
finite discrete-time models can be included in our framework by appropriately
choosing the closed set $\Om$.

\begin{exm}[Final Marginal]
\label{ex:1}
{\rm{In this example $\Om=\cD([0,T];\R_+^d)$.
We fix
a probability measure $\mu$ on $\R_+^d$ with finite $q$-th moments and  set
$$
\cG_\mu
:= \left\{ \gamma(\om)= g(\om(T)) -\mu(g) :
g \in \cC_{q,p}(\R_+^d)\right\},\quad
{\text{where}}\quad 
\mu(g)= \int_{\R_+^d}\ g\, d\mu.
$$
Then,
$\cQ(\cG_\mu)$ consists of
all martingale measures $Q$ whose final marginal
is $\mu$, i.e.,
$$
\E_Q[ h(X_T)] = \mu(h) \quad \mbox{for all } h \in \cB_q(\R_+^d).
$$
}}
\end{exm}

\begin{rem}
{\rm The duality in the setting of Example~\ref{ex:2} with one fixed marginal does not immediately extend to the case of two marginals assuming that $\Omega =\cD([0,T];\mathbb{R}_+^d)$. The difficulty arises from the fact that the coordinate mapping~$X_0$ is not continuous. This issue can be removed by introducing a fictitious element $X_{0-}$ on the Skorokhod space $\cD([0,T];\mathbb{R}_+^d)$, i.e. one considers $\Omega_{x_0-}:=\mathbb{R}_+^d\times\cD([0,T];\mathbb{R}_+^d)$.}
\end{rem}

\begin{exm}[Initial Value and Final Marginal]
\label{ex:2}
{\rm{In addition to a final
marginal, in this example we wish to fix the initial
asset values $x_0 \in \R^d_+$.  However,  
the canonical map,
$X_t : \om \in  \cD([0,T]; \R_+^d) \mapsto \om(t) \in \R_+^d$,
is continuous only for $t=T$ and discontinuous at all other points.
Therefore, 
$\Om_{x_0}:= \left\{\om \in \Om\ :
\ \om(0)=x_0\right\}$
is not an $S$-closed subset of $\cD([0,T]; \R_+^d)$.
To overcome this difficulty, we fix a small time increment $h>0$ and define
$$
\Om_{h,x_0}:= \left\{\om \in \Om :
\om(t)=x_0 \mbox{ for all }  t \in [0,h)\right\}.
$$
One may directly verify that $\Om_{h,x_0}$ is $S$-closed.  We keep $\cG_\mu$ as in the previous example.
Then, the elements
of $\cQ(\cG_\mu)$ restricted to $\Om_{h,x_0}$
are martingale measures
with the final marginal $\mu$ and 
satisfy 
$$
Q(X_t=x_0 \mbox{ for all } t \in [0,h))=1, \quad Q \in \cQ(\cG_\mu).
$$
The set $\cQ(\cG_\mu)$ is non-empty provided that 
$\int x\, \mu(dx)= x_0$. }}
\end{exm}

\begin{exm}[Multiple Marginals]%\label{ex:3}
{\rm{
In Example~\ref{ex:1} we fixed the marginal of the dual measures
at the final time.  In a given application, marginals at other time points 
$\cT=\{t_1, \ldots, t_{N}\}$ might be approximately known. So one may want to fix 
these marginals as well. Since $X_{t_i}$ are  all discontinuous, functions of the form $g(X_{t_i})$ 
are not necessarily $S^*$-continuous  on $\cD([0,T]; \R_+^d)$.
So, as in the previous example, we fix  a small $h>0$
and consider the set given by
$$
\Om_{\cT}:= \bigcap_{i=1}^{N}\ \left\{\om \in \Om\ :
\ X_t(\om)=X_{t_i}(\om) \mbox{ for all }  t \in [t_i,t_i+h)\right\}.
$$
Then, $\Om_{\cT}$ is an $S$-closed subset of 
$\cD([0,T]; \R_+^d)$.
Moreover, for each $i$, $X_{t_i}$ restricted to $\Om_\cT$ 
is $S^*$-continuous.  Given probability measures 
$\{\mu_i\}_{i=1}^N$ on $\mathbb{R}_+^d$ with finite $q$-th moments, we consider the set
$$
\cG_\cT:= \left\{
\gamma(\om)=\sum_{i=1}^N
g_i(X_{t_i}(\om)) - \mu_i(g_i) : 
g_i \in \cC_{q,p}(\R_+^d) \mbox{ for all }i=1,\ldots,N \right\}.
$$
Then, $\cG_\cT \subset \cC_q(\Om_\cT)$.  The measures $Q \in \cQ(\cG_\cT)$ are
martingale measures and have marginal $\mu_i$ at times $t \in [t_i,t_i+h)$.
Assume  $0 \le t_1<\ldots<t_{N}\le T$.
In view of Strassen's result \cite{strassen}, $\cQ(\cG_\cT)$ is non-empty 
if and only if $\mu_i$'s are increasing in convex order, i.e,
$\mu_1(\varphi) \le \ldots \le \mu_N(\varphi)$,
for every convex function $\varphi :\R_+^d \to \R$.}}

\end{exm}

In the following examples, we collect some common option payoffs satisfying the assumptions of Theorem~\ref{t.main}.

\begin{exm}%\label{ex.cont}
{\rm{ The typical
examples of $S^*$-continuous functions are
the payoffs of Asian type options.
Indeed, let $g: [0,T] \to \R$
be continuous.
Then, 
$$
\xi(\om)= \int_0^T g(t) X^i_t(\om)\, dt,
$$
for any $i \in \{1,\ldots,d\}$, is $S^*$-continuous.  
However, the running maximum and minimum of
$X^i$ are only lower and upper semicontinuous, respectively; see \cite{jakubowski}.
We refer the reader to \cite{jakubowski}, \cite{stopo}
for further examples.}}
\qed
\end{exm}

In particular, the duality \eqref{e.dual2} holds for every derivative contract that is a measurable function of the 
underlying assets.

\begin{exm}%\label{ex.meas} 
{\rm 
Since $\Omega$ is a measurable subset of $\mathcal{D}([0,T];\mathbb{R}_+^d)$, we know from \cite{Karandikar1995}
that there exists an $\mathbb{F}$-progressively measurable $d\times d$-matrix-valued process 
$\langle X\rangle=(\langle X\rangle_t)_{t\in [0,T]}$ on $\Omega$ which equals the predictable quadratic 
variation of $X$ $Q$-a.s., for every $\mathbb{F}$-martingale measure $Q$ on $\Omega$. We define the 
$d\times d$-matrix-valued volatility process $\sigma = (\sigma_t)_{t \in [0,T]}$ as the 
square-root of the non-negative, symmetric matrix-valued process
\[
v_t(\omega):= \liminf_{\varepsilon\downarrow 0}\frac{\langle X\rangle_t(\omega)-\langle X\rangle_{(t-\varepsilon)\vee 0}(\omega)}{\varepsilon},\quad (t,\omega)\in[0,T]\times\Omega.
\]
In particular, $\sigma$ is a measurable process $\Omega$. So Theorem~\ref{t.main} yields model-independent 
price bounds for derivative contracts written on $\sigma$. However, the construction of the quadratic variation process 
$\langle X\rangle$ relies on stopping times and therefore, on $\mathbb{F}$-progressively measurable partitions of the 
interval $[0,T]$, which in general are non-deterministic; we refer to \cite{Bichteler1981} for details. 
In particular, derivative contracts depending on $\sigma$ are in general not upper semicontinuous on $\Omega$.}
\end{exm}

As a consequence of Theorem
\ref{t.main}, we obtain a 
{\em{fundamental theorem of asset pricing}} relating the non-emptiness of 
$\cQ(\cG)$ to an appropriate no-arbitrage condition. For classical versions of this result 
see e.g.~\cite{DMW} for discrete time, 
\cite{DS,DS1} for continuous time and the references therein. 
Robust versions have been derived in 
\cite{RFTAP,CKT,DS,mete3}. Our no-arbitrage conditions are 
the following.

\begin{cor}
\label{c.ftap}{\rm{\bf{Robust Fundamental Theorem of Asset Pricing.}}}\\
Under Assumption \ref{asm.1}, the following are equivalent:
\begin{itemize}
\item[{\rm (i)}]  $\cQ(\cG)$ is non-empty.
\item[{\rm (ii)}] $\Phi(\eta;\hcI(\cG))$ is finite for all $\eta \in \cB_p(\Om)$.
\item[{\rm (iii)}] $\Phi(0;\cI(\cG))=0$.
\end{itemize}
\end{cor}
 
\begin{appendix}

\section{Appendix: $S$ and $S^*$-topologies}
\label{s.jakubowski}

The following definition is from Jakubowski \cite{jakubowski,jakubowski4}.

\begin{dfn}%\label{def.adam}
{\rm{For $\{\nu^n\}_{n \in \N} \subset \cD([0,T];\R_+^d)$ and $\nu^* \in \cD([0,T]; \R_+^d)$,
we write $\nu^n\rightharpoonup_S\nu^*$ if for each $\varepsilon>0$, there exist functions
$\{\nu^{{n},\varepsilon}\}_{n\in \N}$ and $\nu^{*,\varepsilon}$ in 
$\cD([0,T]; \R_+^d)$ which are of}} finite variation {\rm{such that
$$
 \|\nu^*-\nu^{*, \varepsilon}\|_\infty \leq \varepsilon, \quad 
  \|\nu^{n}-\nu^{{n},\varepsilon}\|_\infty \leq \varepsilon \quad \mbox{for every } n\in\mathbb{N},
$$
and
\begin{equation}\label{e.s}
  \lim_{n \to \infty} \int_{[0,T]}\ f(t)\, d\nu^{{n},\varepsilon}_t
  = \int_{[0,T]}\ f(t) \, d\nu^{*,\varepsilon}_t ,
\end{equation}
for all $f\in \cC_b([0,T];\mathbb{R}^d)$, where the integrals in \eqref{e.s} are Stieltjies integrals
with $\nu^{{n},\varepsilon}_{0-}= \nu^{*,\varepsilon}_{0-}= 0$.
The topology
generated by this sequential convergence is
called the}} $S$-topology. 
\end{dfn}

In particular, a subset $C \subset \cD([0,T];\mathbb{R}^d_+)$ is
$S$-closed if and only if it is 
sequentially closed for the above notion
of convergence, i.e.,
if $\{\nu^n\}_{n \in \N} \subset C$  
and $\nu^n\rightharpoonup_S\nu^*$,
then  $\nu^* \in C$. Open sets are the complements
of the closed ones.  One may directly verify that this collection of sets
satisfies the definition of a topology.

\begin{rem}
{\rm{
The  ({\em{a posteriori}})
convergence in this topology could be different from 
the {\em{a priori}} convergence
$\rightharpoonup_S$ defined above.  This
definition of a topology is known as 
the Kantorovich--Kisy\'nski recipe; see \cite{kisy}
or \cite[Sections 1.7.18, 1.7.19 on pages 63-64]{engel}.  
In particular, it is discussed in 
\cite[Appendix]{jakubowski4} that $\{\nu^n\}_{n \in \N}$
converges to  $\nu^*$ in the ({\em{a posteriori}})
$S$-topology,
if every subsequence  
$\{\nu^{n_k}\}_{k\in\mathbb{N}}$ has a further subsequence
$\{\nu^{n_{k_l}}\}_{l\in\mathbb{N}}$ such that
$\nu^{n_{k_l}}\rightharpoonup_S\nu^*$.

As a different example, if  one starts with almost-sure convergence as the {\em{a priori}}
convergence (instead of the $\rightharpoonup_S$ 
convergence as above), then  the resulting {\em{a posteriori}} convergence
is the convergence-in-probability; see \cite{jakubowski4}. }}
\end{rem}

The following fact from \cite{jakubowski, jakubowski4}
is an essential ingredient of our continuity proof.
Recall the up-crossings $U^{a,b,i}_t$ of Definition \ref{d.up}.

\begin{prop}[Jakubowski \cite{jakubowski}, Theorem 2.13; \cite{jakubowski4}, Theorem~5.7]
%\label{p.compact}
A subset $K \subset \cD([0,T];\R_+^d)$ is relatively $S$-compact if and only if 
\be
\label{e.compact}
\sup_{\om \in K} \|\om\|_\infty <\infty \quad
{\text{and}}\quad
\sup_{\om \in K} U_T^{a,b,i}(\om)
< \infty \mbox{ for all } a<b \mbox{ and } i = 1. \dots, d.
\ee
\end{prop}
Let us denote the relative topology of $S$ on $\Omega$ again by $S$. It is not known whether 
$(\Omega,S)$ is completely regular. As this property plays an important role in our analysis, 
we regularize $S$ on $\Omega$ analogously to \cite{stopo}.

\begin{dfn}
{\rm{The}} $S^*$-topology {\rm{on $\Omega$ is the coarsest topology making all $S$-continuous 
functions $\xi : \Omega \to \R$ continuous.}}
\end{dfn}

It is clear from this definition that $S^* \subset S$, and a function 
$\xi \colon \Omega \to \mathbb{R}$ is $S^*$-continuous if and only if it is $S$-continuous.
Moreover, $(\Omega, S)$ and $(\Omega, S^*)$ are both Hausdorff, and since
compact sets stay compact if the topology is weakened, every
$S$-compact subset of $\Omega$ is also $S^*$-compact.

The collection of finite intersections of sets of the form
\[
O_{\xi,\varepsilon}(\omega_*):=\left\{ \om \in \Om : |\xi(\om)- \xi(\om_*)| < 1 \right\}
\]
with arbitrary $S$-continuous functions $\xi \colon \Omega \to \mathbb{R}$, form a neighborhood basis 
at $\om_*$. In particular, for any $S^*$-open set $O$
and $\om_* \in O$, there is a neighborhood of $\om_*$ of the form
\[
\bigcap_{k=1}^n  \left\{ \om \in \Om : |\xi_k(\om)- \xi_k(\om_*)| <1\right\},
\]
contained in $O$, where each $\xi_k$ is an $S$-continuous function from $\Omega$ to $\mathbb{R}$. 
For each $k\leq n$, set $\eta_k(\omega) =|\xi_k(\om)- \xi_k(\om_*)| \wedge 1$ and $\eta(\omega)=
\max_{k\leq n} \eta_k(\omega)$. Then, $\eta$ continuously maps $\Omega$ into $[0,1]$ 
and satisfies $\eta(\om_*)=0$ and $\eta(\om)=1$ 
for all $\om \not \in O$. This is the defining property of a completely regular space.
Hence, $(\Omega,S^*)$ is a completely regular Hausdorff space, (T$_{3\frac12}$).  In fact,
it turns out to be perfectly normal.

\begin{lem}%\label{l.matti1}
$(\Omega,S^*)$ is perfectly normal Hausdorff {\rm{(}}T$_{6}${\rm{)}} and a Lusin space.
In particular, every Borel probability measure on $(\Om, S^*)$ is a Radon measure.
\end{lem}

\begin{proof}
It is well-known that the standard $J_1$-topology on the Skorokhod space 
is Polish. Moreover, by \cite[Theorem~2.13 (vi)]{jakubowski}, $S \subset J_1$. 
So, since $\Omega$ is $S$-closed it is also $J_1$-closed. Therefore, if we
denote the relative $J_1$-topology on $\Omega$ again by $J_1$, 
$(\Omega, J_1)$ is still Polish and $S^* \subset J_1$. As a consequence, 
the identity map from $(\Om,J_1)$ to $(\Om,S^*)$ is bijective and continuous, which 
shows that $(\Omega, S^*)$ is a Lusin space.

\cite[Proposition~I.6.1, page 19]{fern} proves that any  completely regular Lusin space is 
perfectly normal.
We note that \cite{fern} uses the terminology ``{\em{Espaces standards}}''
\cite[Definition~I.2.1, page 7]{fern} which is exactly
a Lusin space and the term ``{\em{r\'egulier}}''  as defined on page 18 in \cite{fern}
corresponds to completely regular. 
The reader may also consult page 64 of \cite{dudley} for 
a brief discussion of this implication.

Finally, on a Lusin space, every Borel probability measure is Radon;
see e.g., \cite[p.~122]{schwartz1973}. 
\end{proof}

We also need the following facts about the $S^*$-topology.

\begin{lem}
\label{l.matti}  Every $S^*$-upper semicontinuous function from $\Omega$ to 
$\mathbb{R}$ is the pointwise limit of a decreasing sequence 
of $S^*$-continuous functions, and the family of Suslin functions generated by 
$\cU_b(\Om)$ includes $\cB_b(\Om)$.
\end{lem}

\begin{proof}
The statement about approximation of upper
semicontinuous functions by continuous ones
is proved in \cite[Theorem~3]{tong}. Also, see 
\cite[Theorem~49~(c)]{dellacheriemeyer} or \cite[Page~61]{engel}.

The statement about Suslin functions is proved in 
\cite{BCK} (see the end of the proof of Theorem 2.2).
Alternatively, by Proposition 421L in \cite[page~143]{fremlin}
on any topological space,  every Baire set is a Suslin set.
On perfectly normal Hausdorff spaces, Baire and Borel sets agree
\cite[Proposition~6.3.4]{bogachev}.  Hence, bounded Borel functions
with respect to $S^*$ are Suslin.  
\end{proof}

\begin{rem}%\label{r.matti}
{\rm{
\cite{stopo} contains more results about the $S^*$-topology on $\cD([0,T];\mathbb{R}^d)$.
In particular, the compact sets of $S^*$ and $S$ agree.
Also,
the $S^*$-topology is the strongest topology on the Skorokhod space 
for which the  compactness criteria \reff{e.compact} holds
and the Riesz representation theorem with the $\beta_0$-topology is true.
}}
\end{rem}

\section{Appendix: $\beta_0$-topology}
\label{sec:beta0}

Let $E$ be a completely regular Hausdorff space
and recall that $\cB_0(E)$ is the set of real-valued, bounded, Borel measurable
functions on $E$ that vanish at infinity. Note that any perfectly
normal topology, such as $S^*$ on $\Om$, is completely regular.

For each $\eta \in \cB_0^+(E)$ consider the semi-norm on $\cC_b(E)$ given by,
$$
\|\xi\|_\eta := \|\xi \eta\|_\infty := \sup_{x \in E}  |\xi(x)\eta(x)|.
$$
The $\beta_0$-topology on $\cC_b(E)$ is generated  by the semi-norms $\| . \|_\eta$
as $\eta$ varies in $\cB_0^+(E)$.  
Importantly, the topological dual of $\cC_b(E)$
with the $\beta_0$-topology is the set of all signed Radon measures of bounded total 
variation on $E$; see e.g., \cite[Theorem~3, page 141]{jarchow} or \cite{sentilles} for further 
details on the $\beta_0$-topology.

%\begin{lem}
%\label{l.bcont}
%Suppose that for a 
%given $\Phi :\cC_b(\Om) \to \R$
%there exists
%$\eta \in \cB_0^+(\Om)$ and $c>0$ so that
%$|\Phi(\xi)- \Phi(\zeta)|  \le c \|\xi -\zeta\|_\eta$
%for all $\xi,\zeta \in \cC_b(\Om)$.
%Then,  $\Phi$
%is $\beta_0$-continuous.
%\end{lem}
%\begin{proof}
%The sets of the form
%$B(r,\eta,\xi_0):= \{\xi \in \cC_b(\Om)\ :\
%\|\xi - \xi_0\|_\eta < r \}$,
%with $r>0$, $\eta \in \cB_0^+(\Om)$, 
%$\xi_0 \in  \cC_b(\Om)$, is a basis 
%for the $\beta_0$-topology. Then,
%the continuity of $\Phi$ follows 
%from elementary arguments.
%\end{proof}

\section{Appendix: Martingale Measures}%\label{sec:martingale}

\begin{lemma}
\label{l.martingale}
Let $Q \in {\cal P}(\Omega)$ such that $\sup_{t \in [0,T]} \mathbb{E}_Q[ |X_t|^q] < \infty$ for some $q > 1$ and
$\mathbb{E}_Q[Y \cdot (X_T - X_t)] = 0$ for every $t \in [0,T]$ and all ${\cal F}_t$-measurable 
$Y \in {\cal C}_b(\Omega)^d$. Then, the canonical map $X$ is an $(\mathbb{F}, Q)$-martingale.
\end{lemma}

\begin{proof} 
Fix $t < T$, and denote by ${\cal A}$ the family of all subsets of $\Omega$ that can be written 
as a finite intersection of sets of the form $X^{-1}_{t_j}(B_j)$ for $t_j \le t$ 
and a Borel subset $B_j$ of $\mathbb{R}^d$. Let $i \in \{1, \dots, d\}$. If we can show that
\begin{equation} \label{defce}
\mathbb{E}_Q[\one_A (X^i_T - X^i_t)] = 0 \quad \mbox{for all }  A \in {\cal A}, 
\end{equation}
it follows from a monotone class argument that 
\[
\mathbb{E}_Q[\one_A (X^i_T - X^i_t)] = 0 \quad \mbox{for all } A \in {\cal F}^X_t.
\] 
By uniform integrability and right-continuity of $X$, this implies
\[
\mathbb{E}_Q[\one_A(X^i_T - X^i_t)] = \lim_{\varepsilon \downarrow 0} 
\mathbb{E}_Q[\one_A(X^i_T - X^i_{t + \varepsilon})] = 0 \quad \mbox{for all } A \in {\cal F}_t,
\] 
which proves the lemma.

To show \eqref{defce}, note that for every set $A \in {\cal A}$ of the form 
\[
A = X^{-1}_{t_1}(B_1) \cap \dots \cap X^{-1}_{t_k}(B_k)
\]
for $t_1, \dots, t_k \le t$ and
Borel subsets $B_1, \dots, B_k$ of $\mathbb{R}^d$, there exist bounded 
continuous functions $f^n_j \colon \mathbb{R}^d \to \mathbb{R}$ such that
\begin{equation}\label{EQ1}
\mathbb{E}_Q[\one_A (X^i_T - X^i_t)] = \lim_{n \to \infty}
\mathbb{E}_Q[f^n_1(X_{t_1}) \cdots f^n_k(X_{t_k}) (X^i_T - X^i_t)].
\end{equation}
On the other hand, for all $n$, one has
\begin{equation} \label{EQ2}
\mathbb{E}_Q[f^n_1(X_{t_1}) \cdots f^n_k(X_{t_k})(X^i_T - X^i_t) ] 
= \lim_{\varepsilon \downarrow 0} \mathbb{E}_Q[f^n_1(X^{\varepsilon}_{t_1}) 
\cdots f^n_k(X^{\varepsilon}_{t_k}) (X_T - X_{t + \varepsilon}) ], 
\end{equation}
for the $S$-continuous functions 
\[
X^{\varepsilon}_{t_j} = \frac{1}{\varepsilon} \int_{t_j}^{t_j +\varepsilon} X_{u \wedge T}\, du, \quad j = 1, \dots, d.
\]
Since $f^n_1(X^{\varepsilon}_{t_1}) \cdots f^n_k(X^{\varepsilon}_{t_k})$ is ${\cal F}_{t+ \varepsilon}$-measurable
and belongs to ${\cal C}_b(\Omega)$, it follows from the assumptions that \eqref{EQ1}--\eqref{EQ2} vanish, 
and the proof is complete.
\end{proof}

\end{appendix}

\bibliographystyle{abbrv}
\bibliography{sp}

\end{document}